\documentclass[a4paper,11pt,twoside]{amsart}
\usepackage[english]{babel}
\usepackage[utf8]{inputenc}

\usepackage[a4paper,inner=2.9cm,outer=2.9cm,top=4cm,bottom=4cm,pdftex]{geometry}
\usepackage{fancyhdr}
\usepackage{titlesec}
\usepackage{color}
\usepackage{bold-extra}
\usepackage{ mathrsfs }
\usepackage{enumitem}
\titleformat{\section}{\normalfont\scshape\centering}{\thesection}{1em}{}
  \titleformat{\subsection}{\bfseries}{\thesubsection}{1em}{}
  \titleformat{\subsubsection}{\bfseries}{\thesubsubsection}{1em}{}

\usepackage{bm}
\usepackage{comment}
\usepackage{graphics}
\usepackage{aliascnt}
\usepackage[pdftex,citecolor=green,linkcolor=red]{hyperref}

\usepackage{amsmath}
\usepackage{amsfonts}
\usepackage{amssymb}
\usepackage{amsthm}
\usepackage{comment}
\usepackage{mathtools}

\newtheorem{theorem}{Theorem}[section]
\newtheorem{corollary}[theorem]{Corollary}
\newtheorem{problem}[theorem]{Problem}
\newtheorem{lemma}[theorem]{Lemma}
\newtheorem{proposition}[theorem]{Proposition}
\theoremstyle{definition}
\newtheorem{definition}[theorem]{Definition}
\newtheorem{remark}[theorem]{Remark}
\newtheorem{conjecture}[theorem]{Conjecture}

\numberwithin{equation}{section}

\makeatletter
\renewcommand{\pod}[1]{\allowbreak\mathchoice
  {\if@display \mkern 18mu\else \mkern 8mu\fi (#1)}
  {\if@display \mkern 18mu\else \mkern 8mu\fi (#1)}
  {\mkern4mu(#1)}
  {\mkern4mu(#1)}
}

\renewcommand{\Re}{\textnormal{Re}}

\newcommand\n{\mathbf{n}}

\renewcommand\deg{\mathrm{deg}}

\let\oldpmod\pmod
\renewcommand{\pmod}[1]{\hspace{-0.1cm}\oldpmod {#1}}

\begin{document}
\title{On the Liouville function at polynomial arguments}
\author{Joni Ter\"{a}v\"{a}inen}
\address{Department of Mathematics and Statistics \\
University of Turku, 20014 Turku\\
Finland}
\email{joni.p.teravainen@gmail.com}

\begin{abstract}
Let $\lambda$ denote the Liouville function. A problem posed by Chowla and by Cassaigne--Ferenczi--Mauduit--Rivat--S\'ark\"ozy asks to show that if $P(x)\in \mathbb{Z}[x]$, then the sequence $\lambda(P(n))$ changes sign infinitely often, assuming only that $P(x)$ is not the square of another polynomial. 

We show that the sequence $\lambda(P(n))$ indeed changes sign infinitely often, provided that either (i) $P$ factorizes into linear factors over the rationals; or (ii) $P$ is a reducible cubic polynomial; or (iii) $P$ factorizes into a product of any number of quadratics of a certain type; or (iv) $P$ is any polynomial not belonging to an exceptional set of density zero. 

Concerning (i), we prove more generally that the partial sums of $g(P(n))$ for $g$ a bounded multiplicative function exhibit nontrivial cancellation under necessary and sufficient conditions on $g$. This establishes a ``99\% version'' of Elliott's conjecture for multiplicative functions taking values in the roots of unity of some order. Part (iv) also generalizes to the setting of $g(P(n))$ and provides a multiplicative function analogue of a recent result of Skorobogatov and Sofos on almost all polynomials attaining a prime value. 
\end{abstract}

\maketitle

\tableofcontents

\section{Introduction}

Let $P(x)\in \mathbb{Z}[x]$ be a non-constant polynomial with positive leading coefficient. It is a notorious open problem to determine whether $P(n)$ represents infinitely many primes as $n$ ranges over the natural numbers $\mathbb{N}\coloneqq \{1,2,3,\ldots\}$. An old conjecture of Bunyakovsky  asserts that the obvious necessary condition, namely that $P$ should be irreducible and have no fixed prime divisor (that is, a prime $p$ that divides $P(n)$ for every $n\in \mathbb{N}$), is also sufficient. Dirichlet's theorem tells us that this is true for linear polynomials, but the conjecture remains open for all nonlinear polynomials. 

We consider a weakening of Bunyakovsky's conjecture that nevertheless remains open for all but a few types of nonlinear polynomials. This weaker form asks whether $P(n)$ produces infinitely many values having an \emph{odd} number of prime factors, provided that $P(x)$ is not the square of another polynomial. Defining the Liouville function by $\lambda(n)=(-1)^{\Omega(n)}$, where $\Omega(n)$ is the number of prime factors of $n$ with multiplicities (and for technical reasons we extend it as an even function to $n<0$, with $\lambda(0)\coloneqq 1$), the above question is asking whether $\lambda(P(n))=-1$ holds for infinitely many $n$. We formulate this, and the corresponding assertion for $\lambda(P(n))=+1$, as the following conjecture.

\begin{definition}
We say that a polynomial $P(x)\in \mathbb{Z}[x]$ is \emph{non-square} if $P(x)$ has positive leading coefficient and $P(x)$ is not of the form $cQ(x)^2$ for any integer $c$ and polynomial $Q(x)\in \mathbb{Z}[x]$.
\end{definition}

\begin{conjecture}[Liouville at polynomials]\label{conj_lambda} Let $P(x)\in \mathbb{Z}[x]$ be a non-square polynomial. Let $v\in \{-1,+1\}$. Then $\lambda(P(n))=v$ holds for infinitely many natural numbers $n$.
\end{conjecture}

Conjecture~\ref{conj_lambda} was formulated by Cassaigne, Ferenczi, Mauduit, Rivat and S\'ark\"ozy in~\cite[Conjecture 1]{cfmrs}; earlier Chowla formulated in~\cite{chowla} the same assertion, but in the  significantly stronger form that 
\begin{align}\label{e78}
\sum_{n\leq x}\lambda(P(n))=o(x)    
\end{align}
for any such polynomial. Since~\eqref{e78} contains Chowla's famous conjecture on $k$-point correlations\footnote{In~\cite{chowla}, Chowla stated~\eqref{e78} in full generality, but the term \emph{Chowla's conjecture} is nowadays typically referring to the special case where $P$ is a product of linear factors.} and is wide open for any polynomials with nonlinear irreducible factors, it makes sense to study Conjecture~\ref{conj_lambda} as a stepping stone in its direction. 

Conjecture~\ref{conj_lambda}  is closely related to the parity problem as well as to correlations of the Liouville function; we shall discuss these connections in Section~\ref{sec: connections} after stating our results.

\section{Statement of results}

\subsection{Results for the Liouville function}\label{sub:lambda}

We now state our results on Conjecture~\ref{conj_lambda}. We shall in fact establish more general results concerning $g(P(n))$ for bounded multiplicative functions $g$ in Subsection~\ref{subsec:mult} (see Theorem~\ref{theo_multiplicative-nonpret} and Corollary~\ref{cor_multiplicative}), from which Corollary~\ref{cor_linear} follows as a special case.

\begin{corollary}[Liouville at polynomials factoring over $\mathbb{Q}$]\label{cor_linear}
Let $P(x)$ be a  non-square polynomial, and suppose that $P(x)$ factors into linear factors over $\mathbb{Q}$. Then for either $v\in \{-1,+1\}$ we have $\lambda(P(n))=v$ for infinitely many $n$. Moreover, the number of such $n\leq x$ is $\gg_{P} x$.
\end{corollary}

\begin{theorem}[Liouville at reducible cubics]\label{theo_cubic}
Let $P(x)=x(x^2-Bx+C)$ be a cubic polynomial with $B\geq 0$ and $C\in \mathbb{Z}$. Then for either $v\in \{-1,+1\}$ we have $\lambda(P(n))=v$ for infinitely many $n$. Moreover, the number of such $n\leq x$ is $\gg_{P}\sqrt{x}$.
\end{theorem}

The restriction to $B\geq 0$ here is to simplify the proof somewhat, and it could likely be relaxed. However, the assumption that $P(x)$ has a linear factor is very important in the proof; compare with Problem~\ref{problem_cubic} below.

\begin{theorem}[Liouville at products of quadratic factors]\label{theo_quadprod}
Let $k\geq 1$ and let $P(x)=((x+h_1)^2+1)((x+h_2)^2+1)\cdots ((x+h_k)^2+1)$ with $h_1,h_2,\ldots, h_k$ distinct integers. Then for either $v\in \{-1,+1\}$ we have $\lambda(P(n))=v$ for infinitely many $n$.
\end{theorem}

Theorem~\ref{theo_quadprod} provides the first instance of an infinite family of high degree polynomials that do not split into linear factors for which Conjecture~\ref{conj_lambda} has been verified. 

\subsection{Results for multiplicative functions}\label{subsec:mult}

\subsubsection{Non-pretentious functions}

We shall now state our results on more general bounded multiplicative functions. Let $g\colon \mathbb{N}\to \mathbb{D}\coloneqq \{z\in \mathbb{C}\colon \,\, |z|=1\}$ be a multiplicative function, that is, $g(mn)=g(m)g(n)$ for all $m,n\geq 1$ coprime (and we extend $g$ as an even function to $\mathbb{Z}$, with $g(0)\coloneqq 1$). Suppose that $g$ takes values in the roots of unity of some order. Then the image $g(\mathbb{N})$ is finite, and (if $g$ is completely multiplicative) it consists precisely of the roots of unity of some order $q\in \mathbb{N}$. Now, under reasonable assumptions on $g$ and $P$, we wish to show that the sequence $(g(P(n)))_{n\geq 1}$ is not eventually constant (the case $g=\lambda$ corresponds to progress on Conjecture~\ref{conj_lambda}). The necessary assumption we must make on $P$ and $g$ is the following. 

\begin{definition}[Admissible pairs] Let $P(x)\in \mathbb{Z}[x]$ have positive leading coefficient, and let $g\colon  \mathbb{N}\to \mathbb{D}$ be a multiplicative function taking values in the roots of unity of some order. We say that the pair $(P,g)$ is \emph{admissible} if there exists a prime $p$ such that the sequence $(g(p^{v_p(P(n))}))_{n\geq 1}$ is not constant. 
\end{definition}

Note that if $(P,g)$ is not admissible, then $g(P(n))=g(P(1))$ for all $n$. Therefore, we must exclude all non-admissible polynomials from our considerations. 

One may now pose the following conjecture. 

\begin{conjecture}[Multiplicative functions at polynomials]\label{conj_mult} Let $(P,g)$ be admissible, and let $v\in g(\mathbb{N})$. Then there are infinitely many $n\in \mathbb{N}$ such that $g(P(n))\neq v$.
\end{conjecture}

This conjecture is in a sense best possible, since for $(P,g)$ admissible it is not necessarily the case that $g(P(n))=v$ has infinitely many solutions for \emph{every} $v\in g(\mathbb{N})$. As a counterexample, let $P(n)=k!n+1$ and let $g(p)$ be arbitrary for $p\leq k$ and $g(p)=-1$ for $p>k$, and extend $g$ completely multiplicatively. Then $g(P(n))\in \{-1,+1\}$ always. 

We obtain various partial results towards  Conjecture~\ref{conj_mult} in the following theorems. The first of these notably verifies a ``99\% version'' of Elliott's conjecture~\cite{elliott-book},~\cite[Conjecture 1.3]{tao-teravainen-duke} on correlations of multiplicative functions in a sense that is explained below. See Section~\ref{sec:notation} for the pretentious distance notation $\mathbb{D}(f,g;x)$ involved. 

\begin{theorem}[99\% Elliott conjecture]\label{theo_multiplicative-nonpret} Let $k\in \mathbb{N}$, let $g_1,\ldots, g_k\colon  \mathbb{N}\to \mathbb{D}$ be multiplicative functions, and let $a_1,\ldots, a_k$ and $h_1,\ldots, h_k$ be positive integers with $a_ih_j\neq a_jh_i$  whenever $i\neq j$. Suppose that $g_1$ takes values in the roots of unity of some order $q$ and 
\begin{align}\label{eq26}
\mathbb{D}(g_1,\chi,\infty)= \infty   
\end{align}
for every Dirichlet character $\chi$.   
Then, for some $\delta>0$ (depending only on $a_i,h_i$ and $q$), we have
\begin{align*}
\limsup_{x\to \infty}\left|\frac{1}{x}\sum_{n\leq x}g_1(a_1n+h_1)\cdots g_k(a_kn+h_k)\right|\leq 1-\delta.   
\end{align*}
\end{theorem}

\textbf{Remarks.}
\begin{itemize}

 \item Note that $\delta$ in Theorem~\ref{theo_multiplicative-nonpret} is uniform over all choices of $g_1$ if $q$ is fixed. Our value of $\delta$ is also effectively computable, although it would decay like a tower of exponentials as $k\to \infty$.

    \item Let us explain the terminology ``99\% Elliott conjecture''. In additive combinatorics, there are numerous inverse theorems that characterize all instances where a quantity of interest is $\geq 1-\delta$ times the trivial bound; this called the ``99\% case''. The hope is that the analysis of this case sheds light on the more difficult task of characterizing instances where the quantity is $\geq \delta$ (the ``1\% case''). For example, a $99\%$ inverse theorem for the Gowers norms was established in~\cite[Theorem 1.2]{eisner-tao}, a $99\%$ Balog--Szemerédi--Gowers theorem was proved in~\cite{shao-bsg}, and a $99\%$ version of the two-point Chowla conjecture was proved in~\cite[Corollary 2]{mr-annals}.

    \item As was shown by Matom\"aki, Radziwi\l{}\l{} and Tao in~\cite[Appendix B]{mrt-average}, the original statement of Elliott's conjecture fails in the complex-valued case on a technicality, namely that a multiplicative function can pretend to be  two different Archimedean characters at different scales. However, in Theorem~\ref{theo_multiplicative-nonpret}, no such correction is not necessary, since $g_1$ takes values in the roots of unity of fixed order. Therefore, the result is in line with Elliott's original formulation of his conjecture. 
    
    \item Some special cases of Theorem~\ref{theo_multiplicative-nonpret} have been established earlier. The case $g_j(n)=g_1(n)^{d_j}$ with $d_1+\cdots +d_k\not \equiv 0\pmod q$ follows (in the 1\% case, where  $0<\delta<1$ is arbitrary) from~\cite[Corollary 1.6]{tao-teravainen-duke}. The case  where again $d_1+\cdots+d_k\not \equiv 0\pmod q$, and additionally $q$ is prime and $a_i=1$, was proved earlier by Elliott~\cite{elliott-conj} (in the 99\% case). Theorem~\ref{theo_multiplicative-nonpret} removes any such additional assumptions on the functions $g_i$. We note that in the case $d_1+\cdots +d_k\equiv 0\pmod q$ the methods used in ~\cite{tao-teravainen-duke} or~\cite{elliott-conj} fail; in fact, these approaches are not sensitive to the shifts $h_i$ being distinct, and this assumption becomes relevant only if $d_1+\cdots+d_k \equiv 0\pmod q$.  Note that it is the case of $d_1+\cdots+d_k \equiv 0\pmod q$ (with $q=2$) that importantly includes Corollary~\ref{cor_linear} as a special case.

\end{itemize}

\subsubsection{Pretentious functions}

Theorem~\ref{theo_multiplicative-nonpret} dealt with $g(P(n))$ for $g$ non-pretentious (in the sense of Granville and Soundararajan) and with $P(x)$ factoring into linear factors over $\mathbb{Q}$. 
 We can also deal with pretentious $g$ (for a wider class of polynomials, in fact); see Theorem~\ref{theo_multiplicative-pret} below. By combining Theorem~\ref{theo_multiplicative-nonpret} with Theorem~\ref{theo_multiplicative-pret} below, we see that Conjecture~\ref{conj_mult} is true whenever our polynomial $P$ factors completely. 

\begin{corollary}[Multiplicative functions at polynomials factoring over $\mathbb{Q}$]\label{cor_multiplicative}
Let $P(x)$ be a polynomial factoring into linear factors over $\mathbb{Q}$, and let $g\colon  \mathbb{N}\to \mathbb{D}$ be a multiplicative function for which $(P,g)$ is admissible. Then for any $v\in g(\mathbb{N})$ we have $g(P(n))\neq v$ for infinitely many $n$, and in fact for a positive lower density set of $n$.
\end{corollary}

\textbf{Remarks.}
\begin{itemize}

\item In~\cite{cfmrs}, it was shown that if $P(x)=(a_1x+h_1)\cdots (a_kx+h_k)$ satisfies the additional restrictions $a_i=a$ and  $h_1\equiv \cdots \equiv h_k\pmod a$ and if $g\colon  \mathbb{N}\to \{-1,+1\}$ is non-pretentious, then $g(P(n))=\pm 1$ for infinitely many $n$. This special case turns out to be considerably easier than the general case, since the extra assumption on $P$ implies (just using the pigeonhole principle) that if $g(P(n))$ is eventually constant, then $g(an+b)$ satisfies a recurrence relation, and hence it is periodic. This connection to recurrences disappears when the mentioned condition does not hold; therefore, it seems that even the case $k=2$ was not known in full generality before Tao's work~\cite{tao-chowla} on the two-point Elliott conjecture. Moreover, the method of~\cite{cfmrs} does not seem to produce a strong quantitative lower bound for the number of $n\leq x$ at which $g(P(n))$ takes a given sign.

\item  Using the methods of the recent paper~\cite{klurman-shkredov-xu}, it should be possible to show that if $g\colon   \mathbb{N}\to \{-1,+1\}$ is a Rademacher random multiplicative function, then $g(P(n))$ takes each sign with positive density for a wide class of polynomials $P$ (see also~\cite[Theorem 1.1]{klurman-shkredov-xu} for a result for Steinhaus random multiplicative functions).
\end{itemize}

For stating Theorem~\ref{theo_multiplicative-pret} on the pretentious case, we need the following definition.

\begin{definition} We say that a polynomial $P(x)\in \mathbb{Z}[x]$ with positive leading coefficient has \emph{property $\mathcal{S}$} if there exists a constant $\eta_0>0$ such that for all $q,b\geq 1$ we have
\begin{align*}
 \liminf_{x\to \infty}\frac{1}{x}|\{n\leq x,n\equiv b\pmod q\colon  \,\, P(n)\,\, \textnormal{is}\,\, n\textnormal{-smooth}\}|\geq \frac{\eta_0}{q}.   
\end{align*}

\end{definition}

\begin{theorem}[Multiplicative functions at polynomials, pretentious case]\label{theo_multiplicative-pret} Let $g\colon  \mathbb{N}\to \mathbb{D}$ be a completely multiplicative function satisfying $\mathbb{D}(g,\chi(n)n^{it};\infty)<\infty$ for some Dirichlet character $\chi$ and some real number $t$.  Let $P(x)\in \mathbb{Z}[x]$ be a polynomial such that $(P,g)$ is admissible, and such that $P$ satisfies property $\mathcal{S}$. Then there exists $\delta>0$ (depending on $P$ and $g$) such that
\begin{align}\label{e75}
\limsup_{x\to \infty}\left|\frac{1}{x}\sum_{n\leq x}g(P(n))\right|\leq 1-\delta.    \end{align}
\end{theorem}

Note that the dependence of $\delta$ on $g$ and $P$ here is necessary, as can be seen from the cases $P(n)=n$, $g(n)=(-1)^{v_p(n)}$ and $P(n)=k!n+1$, $g(p)=1-1/p$ for all primes $p$. Also note that if $g$ is assumed to be merely multiplicative, then we can choose $g$ to satisfy $g(p)=1$ for all primes $p$ and $g(n^2)=\lambda(n)$ for all $n\in \mathbb{N}$, and thus~\eqref{e75} with $P(n)^2$ in place of $P(n)$ is at least as difficult as Conjecture~\ref{conj_lambda}.  

To complement this result, we have the following proposition, which follows with minor modifications from a result of Harman~\cite{harman-2008} (which improved on work of Dartyge~\cite{dartyge}) that $n^2+1$ takes $n^{4/5+\varepsilon}$-smooth values with positive lower density.

\begin{proposition}[Smooth values of polynomials]\label{prop_smooth}
Let $P(x)\in \mathbb{Z}[x]$ be a non-constant polynomial with positive leading coefficient. If $P$ is either quadratic or $P$ factors into linear factors over $\mathbb{Q}$, then $P$ has property $\mathcal{S}$.
\end{proposition}

Property $\mathcal{S}$ should certainly be true for all polynomials (in a much stronger form, in fact; see~\cite{martin}), although this has not been proved. But if property $\mathcal{S}$ did fail for $P(x)$, then $P(n)$ might be $n$-smooth for density $0$ of the integers, in which case it would be hard to rule out a situation where there would exist a rapidly growing sequence $(x_k)$ such that $P(n)$ always had an odd number of prime factors from the set $\mathcal{X}\coloneqq \bigcup_{k\geq 1}[x_k,x_k^{1+1/k^2}]$. And if such an $(x_k)$ did exists, defining a completely multiplicative function $g$ by $g(p)=(-1)^{1_{p\in \mathcal{X}}}$,  we would have $g(P(n))=-1$ for all $n$, and $g$ would be pretentious since $\sum_{k\geq 1} k^{-2}<\infty$, so $g(P(n))$ would not hypothetically satisfy the conclusion of Theorem~\ref{theo_multiplicative-pret}.

\subsubsection{Almost all polynomials}

The problem of $g(P(n))$ for $g$ non-pretentious and $P(x)$ not factoring into linear (or quadratic) factors seems challenging for specific polynomials. However, we are able to show that if all the polynomials of degree $d$ are ordered by height (the maximum absolute value of their coefficients), then 100\% of them satisfy Conjecture~\ref{conj_mult}, in fact in a stronger sense.

\begin{theorem}[Multiplicative functions at almost all polynomials]\label{thm_almostall}
Let $d\geq 1$, $q\geq 1$, and let $g\colon  \mathbb{N}\to \mathbb{D}$ be a completely multiplicative function such that $g^{q}= 1$  and $\mathbb{D}(g^j,\chi;\infty)=\infty$ for every Dirichlet character $\chi$ and every $1\leq j\leq q-1$.
\begin{enumerate}[label=\upshape(\roman*)]
    \item  When ordered by height, almost all polynomials $P(x)\in \mathbb{Z}[x]$ of degree $d$ with positive leading coefficient have the property that for every $v\in g(\mathbb{N})$ there exists $n\geq 1$ such that $g(P(n))=v$.\\ 
    \item  For almost all $a\in \mathbb{N}$ in the sense of logarithmic density, the following holds. For every $v\in g(\mathbb{N})$ there exists $n\geq 1$ such that $g(n^d+a)=v$. 
\end{enumerate}
\end{theorem}

\textbf{Remarks.}
\begin{itemize}

\item The restriction to multiplicative functions $g$ with $g^j$ being non-pretentious for all $1\leq j <q$ is a natural one, since together with Hal\'asz's theorem it implies that $g(\mathbb{N})=\mu_q$, where $\mu_q$ is the set of $q$th roots of unity, whereas without such an assumption it appears hard to characterize what the set $g(\mathbb{N})$ looks like. Note also that if we only assume $\mathbb{D}(g^j,\chi;\infty)=\infty$ for $j=1$, then the claim fails. Indeed, take $g=\lambda h$, where $h$ is the completely multiplicative function satisfying $h(2)=i$ and $h(p)=\chi(p)$ for $p\geq 3$, with $\chi$ any real Dirichlet character. Then $g(\mathbb{N})=\{\pm 1,\pm i\}$, but if $P$ takes only odd values, then $g(P(n))=\pm 1$.

\item Recently, Skorobogatov and Sofos~\cite{sofos} made a breakthrough by proving that almost all polynomials $P(x)$, which have no fixed prime divisor, produce at least one prime value. Theorem~\ref{thm_almostall} can be thought of as a multiplicative function analogue of their theorem, where we find solutions to $g(P(n))=v$ as opposed to $1_{\mathbb{P}}(P(n))=1$. Our proof of Theorem~\ref{thm_almostall} is very different from and shorter than that in~\cite{sofos}, but it does not recover the case of the primes. We also note that our proof applies even if we fix the values of $d-1$ of the coefficients of $P(x)$ to have some values in $[-N,N]$ and let the other two coefficients vary in $[-N,N]$. 

\end{itemize}

\subsection{Multivariate polynomials}

We also briefly consider $\lambda(P(x,y))$, where $P(x,y)$ is a two-variable polynomial. Unlike in the case of one-variable polynomials, certain nonlinear two-variable polynomials have been shown to produce infinitely many primes; see for instance~\cite{iwaniec-1973},~\cite{fi},~\cite{hb},~\cite{hb-moroz},~\cite{maynard}.

We can show that $\lambda(P(x,y))$ has infinitely many sign changes for certain polynomials $P(x,y)$ not covered by these theorems.

\begin{theorem}[Liouville at special multivariate polynomials]\label{thm_multivariate} Let $v\in \{-1,+1\}$, $1\leq a,b\leq 100$ and $k\geq 1$. Then we have
\begin{enumerate}[label=\upshape(\roman*)]
    \item  $\lambda(am^3+bn^4)=v$ for infinitely many coprime pairs $(m,n)$.\\
   \item $\lambda(am^2+bn^k)=v$ for infinitely many coprime pairs $(m,n)$.
\end{enumerate}

\end{theorem}

\textbf{Remarks.}
\begin{itemize}
    
    \item The upper bound $100$ in Theorem~\ref{thm_multivariate} could be increased; the proof involves a numerical computation, whose threshold can be raised with more computing time.

    \item Earlier results on $\lambda(P(x,y))$ include that of Helfgott~\cite{helfgott-lambda}, who handled the case where $P(x,y)$ is a binary cubic form; a result of Green, Tao and Ziegler~\cite{gt-linear},~\cite{gt-mobius},~\cite{green_tao_ziegler} who dealt with those $P(x,y)$ that split completely over $\mathbb{Q}$ (see also a result of Matthiesen~\cite{matt} generalizing this to unbounded multiplicative functions); and a result of Frantzikinakis and Host~\cite{fh-jams}, which allows some quadratic as well as linear factors. All of these results in fact show that $\lambda(P(x,y))$ has mean zero whenever $P$ is one of these polynomials. In the case of Theorem~\ref{thm_multivariate}, our methods are elementary and unable to say anything about this much stronger mean $0$ property. 

\end{itemize}

\section{Connections to other problems}\label{sec: connections}

\subsection{Connection to the parity problem}

Finding sign changes in the sequence $\lambda(P(n))$ requires bypassing the parity problem in sieve theory (see~\cite{selberg},~\cite[Section 16.4]{opera}).  This is in contrast to the problem of finding \emph{almost prime} values of polynomials, where sieve theory has been very successful; see for instance~\cite[Chapter 9]{halberstam-richert},~\cite{iwaniec},~\cite{lemkeoliver}. There is also an obstruction for applying the circle method to the problem of $\lambda(P(n))$, since the circle method is unable to deal with binary problems such as $\lambda(n)=\lambda(n+1)=1$. In the proof of Theorem~\ref{theo_multiplicative-nonpret}, we avoid these obstructions by applying Tao's entropy decrement argument~\cite{tao-chowla} that allows one to transfer problems such as $\lambda(n)=\lambda(n+1)=1$ to ``finite complexity'' problems such as $\lambda(n)=\lambda(n+p)=-1$ which are then amenable to (higher order) Fourier analysis. For Theorem~\ref{theo_quadprod}, in turn, we apply the theory of Pell equations among other things to find suitable $n$; this makes the problem much more algebraic in nature and hence avoids the aforementioned difficulties.

\subsection{Connection to Chowla's conjecture}

Another reason why Conjecture~\ref{conj_lambda} seems challenging is its  close connection to Chowla's conjecture on correlations of the Liouville function. Chowla's conjecture~\cite{chowla} asserts that if $a_1,\ldots, a_k, b_1,\ldots, b_k$ are natural numbers with $a_ib_j\neq a_jb_i$ whenever $i\neq j$, then 
\begin{align}\label{eq25}
\sum_{n\leq x}\lambda(a_1n+b_1)\cdots \lambda(a_kn+b_k)=o(x)    
\end{align}
as $x\to \infty$. See ~\cite{tao-chowla},~\cite{tao-teravainen-duke},~\cite{tt-ant} for progress on the conjecture, as well as in the function field setting recent breakthroughs of Sawin and Shusterman~\cite{sawin-shust1},~\cite{sawin-shust2}.

Corollary~\ref{cor_linear} can be viewed as establishing a 99\% model case of Chowla's conjecture. Conversely, progress on Chowla's conjecture can easily be translated to progress on Conjecture~\ref{conj_lambda}. In particular, results of~\cite{tao-chowla},~\cite{tao-teravainen-duke} imply Conjecture~\ref{conj_lambda} with the correct logarithmic density for $\lambda(P(n))=v$, provided that $\deg(P)\leq 3$ and that $P$ factors into linear factors over $\mathbb{Q}$. For $k\geq 4$, we may apply~\cite[Proposition 7.1]{tao-teravainen-duke} to deduce that 
$\lambda((n+1)(n+2)\cdots (n+k))$ takes both values $\pm 1$ with positive lower density,
although the density is no longer known to be $1/2$, as that would correspond to proving the 4-point Chowla conjecture. The proof of~\cite[Proposition 7.1]{tao-teravainen-duke} crucially uses that the shifts $1,2,\ldots, k$ in the product  form an arithmetic progression, and therefore the proof of Corollary~\ref{cor_linear} requires a different approach to handle the case of arbitrary products of linear factors. 

\subsection{Connection to squarefree values of polynomials}\label{sub:square}

Let us also note that  Conjecture~\ref{conj_lambda} is related to, and arguably more difficult than, the problem of finding \emph{squarefree} values of polynomials (which boils down to finding solutions to $\mu^2(P(n))=1$). For this problem one can successfully apply sieve theory, unlike for Conjecture~\ref{conj_lambda}.

The problem of squarefree values of polynomials has attracted a lot of attention, see~\cite{erdos-1953}, \cite{hooley-1967},~\cite{booker-browning},~\cite{helfgott-cubic} for works handling polynomials of degree at most three (or products thereof). However, even for $P(x)=x^4+2$ it remains an open problem to show that $P(n)$ takes infinitely many squarefree values. This limits what can be proved towards Conjecture~\ref{conj_lambda} as well using existing techniques, namely with current technology we can only expect to make progress on Conjecture~\ref{conj_lambda} for polynomials that factor over $\mathbb{Q}$ into factors of degree at most $3$ (and already the cubic case appears challenging, as we cannot use sieve theory). See however the work of Granville~\cite{granville-abc} that shows the existence of squarefree values of polynomials under the $ABC$ conjecture. 

\subsection{Some open problems}

Based on the discussion above, Conjecture~\ref{conj_lambda} seems difficult for high degree polynomials; nevertheless, the low degree cases are not completely resolved yet either. In~\cite{borwein-choi-ganguli}, Borewin, Choi and Ganguli proved that, for any irreducible quadratic polynomial $P(x)\in \mathbb{Z}[x]$, if $\lambda(P(n))_{n\geq N_P}$ has at least one sign change, it must have infinitely many of them, where $N_P$ is an explicit quantity. Therefore, given any such $P(x)$, one may check in finite time that $\lambda(P(n))$ does have infinitely many sign changes. However, such an algorithmic approach does not prove in finite time that \emph{every} quadratic polynomial satisfies Conjecture~\ref{conj_lambda}; thus, it seems natural to formulate the following problem.  

\begin{problem}\label{problem_quad}
Show that for every integer $d\neq 0$ we have $\lambda(n^2+d)=-1$ for infinitely many $n\geq 1$.
\end{problem}

Very recently, this problem was solved by Srinivasan~\cite{srinivasan}. See also~\cite{srinivasan2} for some related results. We note that the proof of Theorem~\ref{theo_quadprod} reduces in the end to showing that $\lambda(n^2+1)$ is not eventually periodic, a result that does not seem to follow from the results in~\cite{srinivasan}.

Lastly, for the case of irreducible cubic polynomials (for which we at least understand the simpler function $\mu^2(P(n))$), we formulate the following model case.

\begin{problem}\label{problem_cubic} Show that $\lambda(n^3+2)=-1$ for infinitely many natural numbers $n$.
\end{problem}

We do not make progress on Problems~\ref{problem_quad} and~\ref{problem_cubic} in the present paper.

\subsection{Acknowledgments} The author is grateful to the referee for a careful reading of the paper and for numerous helpful comments and suggestions. The author thanks Alexander Mangerel, Kaisa Matom\"aki, Jori Merikoski and Will Sawin for helpful comments. The author was supported by a Titchmarsh Fellowship  and European Union's Horizon
Europe research and innovation programme under Marie Sk\l{}odowska-Curie grant agreement No
101058904.

\section{Notation}\label{sec:notation}

If $A$ is a finite, nonempty set, and $f\colon  A\to \mathbb{C}$ is a function, we use the averaging notations
\begin{align*}
\mathbb{E}_{a\in A}f(a)\coloneqq \frac{1}{|A|}\sum_{a\in A}f(a),\quad 
\mathbb{E}_{a\in A}^{\log}f(a)\coloneqq \frac{\sum_{n\in A}\frac{f(a)}{a}}{\sum_{a\in A}\frac{1}{a}},   
\end{align*}
with the latter one defined when $A\subset \mathbb{N}$.

We employ the usual Landau and Vinogradov asymptotic notations $O(\cdot),o(\cdot), \ll, \gg , \asymp$. 

For two multiplicative functions $f,g\colon  \mathbb{N}\to \mathbb{D}$ (where $\mathbb{D}$ is the closed unit disc of $\mathbb{C}$), we denote their pretentious distance by
\begin{align*}
\mathbb{D}(f,g;x)^2\coloneqq \Big(\sum_{p\leq x}\frac{1-\Re(f(p)\overline{g(p)})}{p}\Big)^{1/2}.    
\end{align*}
We define  $\mathbb{D}(f,g;y,x)$ similarly, but with the summation over $y\leq p\leq x$. For convenience, we extend all our multiplicative functions as even functions to $\mathbb{Z}$, with $g(0)\coloneqq 1$. 

We let $v_p(n)$ denote the largest integer $k$ such that $p^k\mid n$. The functions $\Lambda, \lambda, \mu, \varphi$ are the usual von Mangoldt, Liouville, M\"obius and Euler totient functions. By $\mathbb{Z}_N$ we denote the cyclic group of order $N$. By $\mu_q$ we denote the set of roots of unity of order $q$. The notation $\|x\|$ stands for the distance from $x$ to the nearest integer(s).

\section{The non-pretentious case}

The proof of Theorem~\ref{theo_multiplicative-nonpret} proceeds in four steps, presented in the next four subsections. 

In step 1, we apply the entropy decrement argument to reduce matters to obtaining a nontrivial bound for the two-dimensional correlation
\begin{align}\label{e168}
 \mathbb{E}_{p\leq D}\overline{g_1\cdots g_k(p)}\mathbb{E}_{n\leq x}g_1(a_1n+ph_1)\cdots g_k(a_kn+ph_k)   
\end{align}
for some large constant $D$. This is easier to analyze than the original correlation.

In step 2, we apply the dense model theorem and  the pseudorandomness of the von Mangoldt function to upper bound~\eqref{e168} with an average of the form
\begin{align}\label{e169a}
\mathbb{E}_{d\leq D'}|\mathbb{E}_{n\leq x}g_1(a_1n+dh_1'+r_1)\cdots g_k(a_kn+dh_k'+r_k)|     
\end{align}
for some large constant $D'$. Note that there is no prime weight in this correlation.

In step 3, we show that~\eqref{e169a} can be close to the maximum possible value only if $g_1$ is close to a polynomial phase; this is a variant of the 99 \% inverse theorem for the Gowers norms. The conclusion after this step is, roughly speaking, that Theorem~\ref{theo_multiplicative-nonpret} holds unless for any fixed $\delta'>0$ for almost all $x$ there exist a polynomial $P_x$ of degree $\leq k-1$ such that
\begin{align}\label{eqq1}
|\mathbb{E}_{x\leq n\leq x+D'}g_1(n)e(-P_x(n))|\geq 1-\delta'.    
\end{align}

In step 4, we use a Fourier-analytic argument to show that any sequence taking values in the $q$th roots of unity cannot correlate too strongly with minor arc polynomial phases, and that (by the Matom\"aki--Radziwi\l{}\l{} theorem) the multiplicative function $g_1$ cannot locally correlate with major arc polynomial phases either. This then contradicts~\eqref{eqq1}.

In this section, we think of $k, a_1,\ldots, a_k,h_1,\ldots, h_k,q\in \mathbb{N}$ as being fixed and allow constants to depend on them. However, all estimates will be uniform over the multiplicative functions $g_i$ satisfying the conditions of Theorem~\ref{theo_multiplicative-nonpret}.

\subsection{Reduction to two-dimensional correlations}

Our goal in this subsection is to reduce the proof of Theorem~\ref{theo_multiplicative-nonpret} to obtaining a nontrivial bound for~\eqref{e168}.

We first need the following (consequence of the) entropy decrement argument of Tao, which is a slight modification of the statements in~\cite{tt-chowla},~\cite{tao-teravainen-duke}.

\begin{lemma}[Entropy decrement argument]\label{le_entropy} Let $a_1,\ldots, a_k$, $h_1,\ldots, h_k\in \mathbb{N}$ be fixed.  Then, uniformly for any multiplicative functions  $g_1,\ldots, g_k\colon  \mathbb{N}\to \mathbb{D}$ and any $2\leq \omega\leq x/2$, $2\leq D\leq (\log \omega)^{1/2}$, we have
\begin{align}\label{e124a}\begin{split}
    &\mathbb{E}_{x/\omega\leq n\leq x}^{\log}g_1(a_1n+h_1)\cdots g_k(a_kn+h_k)\\
    &=\mathbb{E}_{p\leq D}^{\log}\overline{g_1\cdots g_k(p)}\mathbb{E}_{x/\omega\leq n\leq x}^{\log}g_1(a_1n+ph_1)\cdots g_k(a_kn+ph_k) +o_{D\to \infty}(1).
    \end{split}
\end{align}
\end{lemma}

\begin{proof}
We first reduce to the case where the $g_i$ are almost unimodular in the sense that\footnote{Note that since the value of $\delta$ in Theorem~\ref{theo_multiplicative-nonpret} is required to be uniform over $g_i$, we do need a short argument to reduce to unimodular functions.}
\begin{align}\label{e125a}
\sum_{p}\frac{1-|g_1\cdots g_k(p)|}{p}<\infty.    
\end{align}
Applying the triangle inequality repeatedly, we see that~\eqref{e125a} follows if we show that
\begin{align}\label{e112}
\sum_{p}\frac{1-|g_j(p)|}{p}<\infty   
\end{align}
for each $1\leq j\leq k$. By the triangle inequality and Shiu's bound~\cite{shiu}, we have uniformly for $1\leq m_j\leq x^{1/2}$ the very crude estimate
\begin{align}\label{e123a}\begin{split}
|\mathbb{E}_{n\leq x}g_1(a_1n+m_1)\cdots g_k(a_kn+m_k)|&\leq  \mathbb{E}_{n\leq x}|g_j(a_jn+m_j)|\\
&\leq a_j \mathbb{E}_{n\leq a_jx}|g_j(n)|+o(1)\\
&\ll \prod_{p\leq a_j x}\left(1-\frac{1-|g_j(p)|}{p}\right)+o(1). 
\end{split}
\end{align}
Thus, if~\eqref{e112} does not hold, then the correlation average on the left of~\eqref{e123a} is $o(1)$, which by partial summation (and taking $m_i=h_i$ or $m_i=ph_i$) implies that both of the logarithmic correlation averages in~\eqref{e124a} are $o(1)$, in which case there is nothing to prove. Hence, we may assume that~\eqref{e125a} holds.

Now, using the triangle inequality together with~\eqref{e125a} and $g_j(p)g_j(a_jn+h_j)=g_j(a_jpn+ph_j)$ for $p\nmid a_jn+h_j$, the claim~\eqref{e124a} reduces to showing that for any $\varepsilon>0$ we have
\begin{align}\label{e163}
|\mathbb{E}_{p\leq D}^{\log}\mathbb{E}_{x/\omega\leq n\leq x}^{\log}&\left(g_1(a_1pn+ph_1)\cdots g_k(a_kpn+ph_k)-g_1(a_1n+ph_1)\cdots g_k(a_kn+ph_k)\right)|\leq \varepsilon    
\end{align}
for $D\geq D_0(\varepsilon)$. Using the dilation invariance property of logarithmic averages, that is,
\begin{align*}
\mathbb{E}_{x/\omega\leq n\leq x}^{\log}a(pn)=\mathbb{E}_{x/\omega\leq n\leq x}^{\log}a(n)p1_{p\mid n}+O\left(\frac{\log p}{\log \omega}\right)    
\end{align*}
for $1$-bounded sequences $a(n)$,~\eqref{e163} reduces to
\begin{align}\label{eq163}
|\mathbb{E}_{p\leq D}^{\log}\mathbb{E}_{x/\omega\leq n\leq x}^{\log}g_1(a_1n+ph_1)\cdots g_k(a_kn+ph_k)(p1_{p\mid n}-1)|\leq \varepsilon/2    
\end{align}
for $D_0(\varepsilon)\leq D\leq (\log \omega)^{1/2}$. But this follows from the entropy decrement argument in~\cite[Section 4 and Theorem 3.1]{tt-chowla} (the argument there is formulated for correlations of $\lambda$, but the same argument works for the correlations of any $1$-bounded multiplicative functions $g_1,\ldots, g_k$, uniformly in the choice of functions; cf.~\cite[Section 3]{tao-teravainen-duke}. Note also that the correlation averages in~\cite{tt-chowla} are logarithmic, but this assumption is not used after reducing matters to two-dimensional correlations). This completes the proof.
\end{proof}

Note that Lemma~\ref{le_entropy} is a statement about logarithmic averages --- and in fact it has to be since one can show that the direct analogue of~\eqref{e124a} with the average $\mathbb{E}_{n\leq x}$ fails (to see this, take $k=1$ and $g_1(n)=n^{it}$ with $t\neq 0$). However, the following simple lemma (with $f_1(n)=g_1(a_1n+h_1)\cdots g_k(a_kn+h_k)$ and $f_2(n)=\mathbb{E}_{p\leq D}^{\log}\overline{g_1\cdots g_k(p)}g_1(a_1n+h_1)\cdots g_k(a_kn+h_k)$) shows that we can still use Lemma~\ref{le_entropy} to estimate the asymptotic averages we are interested in.

\begin{lemma}\label{le_f1f2}
Let $10\leq \omega\leq x/10$. Let $f_1,f_2\colon  \mathbb{N}\to \mathbb{D}$ be any functions. Then 
\begin{align*}
1-\left|\mathbb{E}_{n\leq x}f_1(n)\right|\geq \frac{\log \omega}{2\omega} \left(1-\max_{x'\in [x/\omega,x]}\left|\mathbb{E}_{n\leq x'}f_2(n)\right|-\left|\mathbb{E}_{x/\omega\leq n\leq x}^{\log}f_1(n)-f_2(n)\right|-\frac{8}{\log \omega}\right).   \end{align*}
\end{lemma}

\begin{proof}
Since $|z|=\max_{\phi\in \mathbb{R}}\textnormal{Re}(e(\phi)z)$, for some real number $\theta$ we have
\begin{align*}
1-\left|\mathbb{E}_{n\leq x}f_1(n)\right|=\mathbb{E}_{n\leq x}\textnormal{Re}(1-e(\theta)f_1(n))\geq \frac{1}{x}\sum_{x/\omega\leq n\leq x}\textnormal{Re}(1-e(\theta)f_1(n))-\frac{2}{\omega}.    
\end{align*}
Since $1/x\geq 1/(n\omega)$ for $n\in [x/\omega,x]$ and $\sum_{x/\omega\leq n\leq x}1/n\geq \log \omega-1\geq (\log \omega)/2$, we have
\begin{align*}
\frac{1}{x}\sum_{x/\omega\leq n\leq x}\textnormal{Re}(1-e(\theta)f_1(n))&\geq \frac{\log \omega}{2\omega}\mathbb{E}_{x/\omega\leq n\leq x}^{\log}\textnormal{Re}(1-e(\theta)f_1(n))\\
&\geq \frac{\log \omega}{2\omega}\left(\mathbb{E}_{x/\omega\leq n\leq x}^{\log}\textnormal{Re}(1-e(\theta)f_2(n))-\left|\mathbb{E}_{x/\omega\leq n\leq x}^{\log}f_1(n)-f_2(n)\right|\right).    
\end{align*}
Lastly, note that by the triangle inequality and partial summation
\begin{align*}
\mathbb{E}_{x/\omega\leq n\leq x}^{\log}\textnormal{Re}(1-e(\theta)f_2(n))&\geq 1-\mathbb{E}_{x/\omega\leq n\leq x}^{\log}|f_2(n)|\\
&\geq 1-\max_{x'\in [x/\omega,x]}\frac{1+\log \omega}{\sum_{x/\omega\leq n\leq x}\frac{1}{n}}\mathbb{E}_{n\leq x'}|f_2(n)|\\
&\geq 1-\max_{x'\in [x/\omega,x]}\mathbb{E}_{n\leq x'}|f_2(n)|-\frac{4}{\log \omega}.
\end{align*}

\end{proof}

By Lemmas~\ref{le_entropy} and~\ref{le_f1f2}, if for some $\eta>0$ we prove 
\begin{align}\label{e166}
\limsup_{D\to \infty}\limsup_{x\to \infty}|\mathbb{E}_{p\leq D}\overline{g_1\cdots g_k(p)}\mathbb{E}_{n\leq x}g_1(a_1n+ph_1)\cdots g_k(a_kn+ph_k)|\leq 1-\eta,
\end{align}
 then it follows that
\begin{align*}
|\mathbb{E}_{n\leq x}g_1(a_1n+h_1)\cdots g_k(a_kn+h_k)|\leq 1-\frac{\eta\log \omega}{4\omega}.     
\end{align*}

The expression~\eqref{e166} is now a two-dimensional correlation average, so step 1 of the proof is complete.

\subsection{Removing the prime weight}

We shall now execute step 2 of the proof, reducing matters to obtaining a nontrivial bound for an expression of the form~\eqref{e169a}. We must therefore replace the prime average in~\eqref{e166} with an average over the integers.

\begin{lemma}[Removing the prime weight]\label{le_primeweight} Let $k\geq 1$ and $a_1,\ldots, a_k$, $h_1,\ldots, h_k\in \mathbb{N}$ be fixed with $(a_i,h_i)$ linearly independent over $\mathbb{Q}$. Also let $\varepsilon>0$ and let $D_1$ be large enough in terms of $\varepsilon$. Then there exist $D_2\in[D_1^{1/2},D_1]$, $R=R(\varepsilon)$ and $1\leq h_i'\leq R(\varepsilon)$ with $(a_i,h_i')$ linearly independent over $\mathbb{Q}$ and such that the following holds. For any $g_1,\ldots, g_k\colon  \mathbb{N}\to\mathbb{D}$ and any $x$ large enough, for some $1\leq r_i\leq R(\varepsilon)$ we have
\begin{align}\label{e169}
\mathbb{E}_{p\leq D_1}\Big|\mathbb{E}_{n\leq x}\prod_{j=1}^k g_j(a_jn+ph_j)\Big|\leq \mathbb{E}_{d\leq D_2}\Big|\mathbb{E}_{n\leq x}\prod_{j=1}^k g_j(a_jn+dh_j'+r_j)\Big|+\varepsilon.
\end{align}
\end{lemma}

\begin{proof}
In what follows, let 
\begin{align*}
D\gg w\gg \varepsilon^{-1}\gg 1    
\end{align*}
with each parameter large enough in terms of the ones to the right of it.

Consider the average
\begin{align}\label{e171}
 \mathbb{E}_{p\leq D}c(p)\mathbb{E}_{n\leq x}\prod_{j=1}^k g_j(a_jn+ph_j)  
\end{align}
for $x$ large enough in terms of $D$ and for any function $c\colon  \mathbb{N}\to \mathbb{D}$ (which may depend on all the parameters). Since the von Mangoldt function and the indicator function of the primes are interchangeable in the sense that
\begin{align*}
\mathbb{E}_{d\leq y}|\Lambda(d)-(\log y)1_{\mathbb{P}}(d)|=o_{y\to \infty}(1),    
\end{align*} 
we can write~\eqref{e171} as
\begin{align}\label{e167}
\mathbb{E}_{d\leq D}\Lambda(d)c(d)\mathbb{E}_{n\leq x}g_1(a_1n+dh_1)\cdots g_k(a_kn+dh_k)+o_{D\to \infty}(1). 
\end{align}
To analyze this further, we introduce the $W$-trick. Let $W=\prod_{p\leq w}p$, and recall that $D$ is large in terms of $w$. By splitting the $d$ average in~\eqref{e167} into progressions modulo $W$, we can write it as
\begin{align}\label{e172}
\mathbb{E}_{1\leq b\leq W,(b,W)=1}\mathbb{E}_{d'\leq D/W}\Lambda(Wd'+b)c(Wd'+b)\mathbb{E}_{n\leq x}\prod_{j=1}^k g_j(a_jn+(Wd'+b)h_j)+o_{w\to \infty}(1).
\end{align}
Now, since $d\mapsto \frac{\varphi(W)}{W}\Lambda(Wd+b)c(Wd+b)$ is majorized by the pseudorandom measure\footnote{Strictly speaking, the measure becomes pseudorandom in the sense of~\cite[Section 6]{gt-linear} only if $w$ is a function of $x$ tending to infinity; however, as $w$ can be taken to be an arbitrarily large absolute constant, everything works similarly.} $\frac{\varphi(W)}{W}\Lambda(Wn+b)$, by the dense model theorem~\cite[Proposition 10.3]{gt-linear}, we may write 
\begin{align*}
\frac{\varphi(W)}{W}\Lambda(Wd+b)c(Wd+b)=c_{b,W}(d)+E_{b,W}(d), \end{align*}
where $|c_{b,W}(d)|\leq 1$ and $\|E_{b,W}(d)\|_{U^k[D/W]}=o_{w\to \infty}(1)$. Splitting the $n$ average in~\eqref{e172} into short intervals of length $D/W$ and applying the generalized von Neumann theorem~\cite[Proposition 7.1]{gt-linear}, we see that the contribution of $E_{b,W}(d)$ is negligible, so~\eqref{e172} can be written as
\begin{align}\label{e173}
\mathbb{E}_{1\leq b\leq W,(b,W)=1}\mathbb{E}_{d'\leq D/W}c_{b,W}(d')\mathbb{E}_{n\leq x}\prod_{j=1}^k g_j(a_jn+(Wd'+b)h_j)+o_{w\to \infty}(1).    
\end{align}
Since $\varepsilon^{-1}$ grows slowly in terms of $w$, we may assume that the $o_{w\to \infty}(1)$ term above is bounded by $\varepsilon$ in modulus for $D/W$ large enough in terms of $\varepsilon$. 
Now, if $R(\varepsilon)$ is large enough in terms of $\varepsilon$, then for $h_{j,W}\coloneqq Wh_j$ we have $1\leq h_{j,W}\leq R(\varepsilon)$ and the vectors $(a_i,h_{i,W})$ are linearly independent over $\mathbb{Q}$. Therefore, each of the inner averages in~\eqref{e173} is 
\begin{align*}
\leq  \mathbb{E}_{d\leq D/W}\Big|\mathbb{E}_{n\leq x}\prod_{j=1}^k g_j(a_jn+dh_{j,W}+bh_j)\Big|.   
\end{align*}
Note finally that $D/W\geq D^{1/2}$ since $w$ is growing slowly enough in terms of $D$. The claim now follows by averaging over $b$ and applying the pigeonhole principle.
\end{proof}

We now conclude that the proof of Theorem~\ref{theo_multiplicative-nonpret} reduces to the following. Let $a_i,q\in \mathbb{N}$ be fixed.  Then for some constant $\delta'>0$ (depending only on the above parameters) the following holds. Whenever $h_i\in \mathbb{N}$ satisfy $a_ih_j\neq a_jh_i$ for $i\neq j$, for any $r_i\in \mathbb{N}$ we have the bound
\begin{align}\label{e176}
\limsup_{D'\to \infty}\limsup_{x\to \infty}\mathbb{E}_{d\leq D'}|\mathbb{E}_{n\leq x} g_1(a_1n+dh_1+r_1)\cdots g_k(a_kn+dh_k+r_k)|\leq 1-\delta'.  \end{align}

We shall show that~\eqref{e176} in turn follows from the following two propositions.

\begin{proposition}[A 99\% inverse theorem for averages along arithmetic progressions] \label{prop_99} Let $N$ be a prime, and let $f_1,\ldots, f_k\colon  \mathbb{Z}_N\to \mathbb{D}$ be any functions. Let $h_1,\ldots, h_k\in \mathbb{N}$ be distinct and fixed, let $\eta>\varepsilon>0$, and suppose that
\begin{align}\label{e143}
\mathbb{E}_{d\in \mathbb{Z}_N, 0\leq d\leq \varepsilon N}|\mathbb{E}_{n\in \mathbb{Z}_N}f_1(n+dh_1)\cdots f_k(n+dh_k)|\geq 1-\eta.    
\end{align}
Then, there exist some constants $c_{\varepsilon,\eta}$, $\gamma(k)>0$ such that for any  partition $\mathcal{I}$ of $\mathbb{Z}_N$ into intervals\footnote{We define an interval $[A,B]\subset \mathbb{Z}_N$ with $A,B\in \mathbb{R}$ as the image of $[\lceil A\rceil, \lfloor B \rfloor]\cap \mathbb{Z}$ under this canonical projection $x\mapsto x\pmod N$ from $\mathbb{Z}\to \mathbb{Z}_N$.} of length $\in [c_{\varepsilon, \eta} N/2,c_{\varepsilon,\eta}N]$ we have
\begin{align}\label{e142}
\frac{1}{|\mathcal{I}|}\sum_{I\in \mathcal{I}}\sup_{\substack{P(Y)\in \frac{1}{N}\mathbb{Z}[Y]\\\deg{P}\leq k-1}}|\mathbb{E}_{n\in I}f_1(n)e(-P(n))|\geq 1-2\eta^{\gamma(k)}.    
\end{align}
\end{proposition}

\textbf{Remarks.}

\begin{itemize}
\item Note that we want the $d$ average in Proposition~\ref{prop_99} to be much shorter than the $n$ average so that we can pass between cyclic groups and the integers without losing nontrivial factors. However, the fact that the $d$ average is shorter means that if we wanted to upper bound~\eqref{e143} with the Gowers norm $\|f_1\|_{U^{k-1}(\mathbb{Z}_N)}$ using the generalized von Neumann theorem, we would lose a vital constant factor. For this reason we have not been able to reduce Proposition~\ref{prop_99} directly to the "99\% inverse theorem" of Eisner and Tao~\cite{eisner-tao}. +Moreover, the bound in~\eqref{e142} is polynomial in $\eta$, whereas the corresponding bound in~\cite{eisner-tao} involves an unspecified function $1-o_{\eta \to 0}(1)$. 

\item We note that the primality assumption on $N$ is needed, since if $r\mid N$ and $r=O(1)$ is suitably chosen, there exist $r$-periodic unimodular functions on $\mathbb{Z}_N$ which are not of the form $e(P(n))$, which would be an issue in the case $h_j=jr$.

\item The reason we need to localize in~\eqref{e142} into subintervals of $\mathbb{Z}_N$ is that slowly-varying functions such as $f(n)=e(\log(1+(n\pmod N)/N))$ (which is a function on $\mathbb{Z}_N$) satisfy $f(n+d)=f(n)+O(\eta)$ whenever $n\leq N-\eta N$ and  $d\leq \eta N$, despite not being globally close to a polynomial phase on $\mathbb{Z}_N$. This is a phenomenon that does not arise in the usual 99\% inverse theorem for the Gowers norms.
\end{itemize}

\begin{proposition}[Discrete-valued multiplicative functions do not strongly correlate with locally polynomial phases.]\label{prop_polyphase} Let $q\geq 2$, $k\geq 1$ be fixed. Let $g\colon  \mathbb{N}\to \mu_q$ be multiplicative, where $\mu_q$ is the set of $q$th roots of unity. Suppose that $g$ satisfies the non-pretentiousness assumption $\mathbb{D}(g,\chi;\infty)=\infty$ for every Dirichlet character $\chi$. Then there exists $\delta=\delta_{q}>0$ such that, whenever $X\gg H\gg Q\geq 1$ and each parameter is large enough in terms of the ones to the right of it, we have
\begin{align*}
\sup_{\substack{P(Y)\in \mathbb{R}[Y]\\\deg{P}\leq k}}\,\sup_{\substack{I\subset [x,x+H]\\|I|\geq H/Q}}\,\sup_{a,r\leq Q}|\mathbb{E}_{n\in I,n\equiv a\pmod r}g(n)e(-P(n))|\leq 1-\delta
\end{align*}
for all but at most $X/\log \log H$ values of $x\in [1,X]\cap \mathbb{Z}$.
\end{proposition}

\begin{proof}[Proof of Theorem~\ref{theo_multiplicative-nonpret} assuming Propositions~\ref{prop_99} and~\ref{prop_polyphase}] 

Recall that it suffices to prove~\eqref{e176}. By denoting $A=a_1\cdots a_k$ and splitting the $d$ average in~\eqref{e176} into progressions modulo $A$, and applying the pigeonhole principle, we can bound the left-hand side of~\eqref{e176} by
\begin{align}\label{e180}
\mathbb{E}_{d\leq D'/A}\Big|\mathbb{E}_{n\leq x}\prod_{j=1}^k g_j\big(a_j(n+d(A/a_j)\cdot h_j)+r_j'\big)\Big|+o_{D'\to \infty}(1)     
\end{align}
for some $1\leq r_j'\leq r_j+Ah_j$. Let $\varepsilon>0$ be a small constant such that $D'$ is large in terms of $\varepsilon^{-1}$. Denote $D=\varepsilon^{-1}D'/A$. Introduce the functions
\begin{align*}
g_{j;\ell,x}(n)\coloneqq g_j(\ell(n+\lfloor x\rfloor)+r_j').
\end{align*}
Denote $h_j'=A/a_j\cdot h_j$. We split the $n$ average in~\eqref{e180} into intervals of length $D$, thus bounding~\eqref{e180} by
\begin{align}\label{e139}
\frac{1}{X}\int_{1}^{X}\mathbb{E}_{d\leq \varepsilon D}\Big|\mathbb{E}_{n\leq D}\prod_{j=1}^k  g_{j;a_j,x'}(n+dh_j')\Big|\, dx'+o_{D'\to \infty}(1).
\end{align}

Next, we shall embed the average in~\eqref{e139} into a cyclic group. Let $M=\max_i h_i'$, let $\varepsilon>0$ be small in terms of $M$, and let $N\in [(1+\varepsilon M)D,(1+2\varepsilon M)D]$ be a prime; such a prime exists by the prime number theorem as long as $D$ is large enough in terms of $\varepsilon$. We define the functions $f_j=f_{j,x'}\colon  \mathbb{Z}_N\to \mathbb{D}$ by $f_j(n)= g_{j;a_j,x'}(n)$ for $0\leq n<N$, and extend them to be $N$-periodic (and thus they can be interpreted as functions on $\mathbb{Z}_N$). Since $\varepsilon M$ is small, wraparound issues are negligible, so we can bound the double average in~\eqref{e139} with
\begin{align}\label{e181}
\mathbb{E}_{d\in \mathbb{Z}_N, 0\leq d\leq \varepsilon N}|\mathbb{E}_{n\in \mathbb{Z}_N}f_{1,x'}(n+dh_1')\cdots f_{k,x'}(n+dh_k')|+o_{\varepsilon\to \infty}(1).    
\end{align}
We are now in a position to apply Proposition~\ref{prop_99}. Since $\mathbb{D}(g_1,\chi;\infty)=\infty$ for all $\chi$ and since $g_1$ takes values in $\mu_q$, there exists some $\delta_0=\delta_{0,q}>0$ such that whenever $\eta>0$ and $N\geq N_0(\eta,k)$ we have
\begin{align}\label{e177}
\sup_{\substack{P(Y)\in \mathbb{R}[Y]\\\deg{P}\leq k}}\sup_{\substack{I\subset [1,N]\\|I|\geq \eta N}} |\mathbb{E}_{n\in I}\,  g_{1;a_j,x'}(n)e(-P(n))|\leq 1-\delta_0   
\end{align}
for all $x'\leq X$, apart from $\leq \eta X$ exceptions. 

But now Proposition~\ref{prop_99} shows that~\eqref{e181} is $\leq 1-2\delta'$ for some constant $\delta'>0$ depending only on $\delta_0$, unless $x'\leq X$ belongs to an exceptional set of measure $\leq \eta X$. Taking $\eta>0$ small in terms of $\delta'$ and recalling that~\eqref{e139} is upper bounded by an average over $x'$ of expressions of the form~\eqref{e181}, this gives the desired bound for~\eqref{e176}.
\end{proof}

\subsection{A 99\% inverse theorem}\label{subsec:inverse}

We will next prove Proposition~\ref{prop_99}; this corresponds to step 3 of the proof of Theorem~\ref{theo_multiplicative-nonpret}. The proof will be based on a favorable change of coordinates and repeated use of discrete differentiation and integration. 

\begin{proof}[Proof of Proposition~\ref{prop_99}]
In what follows, we allow the implied constants in any estimates depend on $k,h_1,\ldots, h_k$. We may assume that $N$ is large enough in terms of $\varepsilon,\eta $, since otherwise (by choosing the function $c_{\varepsilon,\eta}$ appropriately) we have $c_{\varepsilon,\eta}N<1$. We may also assume that $\gamma(k)>0$ is small enough and that $\eta>0$ is small in terms of $k$.\\

\textbf{Reduction to unimodular functions.} We first reduce to the case where the $f_i$ are unimodular. Suppose that the unimodular case has been proved. If there exists $i$ for which
\begin{align}\label{eq6}
|\{m\in \mathbb{Z}_N\colon  |f_i(m)|\leq 1-\eta^{1/2}\}|\geq 2\eta^{1/2}N,    
\end{align}
then the left-hand side of~\eqref{e143} is trivially $<1-\eta$ in absolute value (since the number of pairs $(n,d)$ satisfying $m=n+h_id$ with $n\in \mathbb{Z}_N$ and $0\leq d\leq \varepsilon N$ is $\varepsilon N+O(1)$). Therefore, we may assume that~\eqref{eq6} fails for every $i$. Then we may write $f_i(n)=\tilde{f_i}(n)+O(\eta^{1/2})+E(n)$, where\footnote{For $z=0$, we can define $z/|z|$ arbitrarily, say as $1$.} $\tilde{f_i}(n)=f_i(n)/|f_i(n)|$ is unimodular and $|E(n)|\leq 2$, and $E(n)=0$ for all but $\ll \eta^{1/2} N$ values of $n\in \mathbb{Z}_N$. Substituting this formula for $f_i(n)$ to~\eqref{eq6} for each $i$, we see that the left-hand side of~\eqref{e143} is 
\begin{align*}
\mathbb{E}_{d\in \mathbb{Z}_N,d\leq \varepsilon N}| \mathbb{E}_{n\in \mathbb{Z}_N}\tilde{f_1}(n+h_1d)\cdots \tilde{f_k}(n+h_kd)|+O(\eta^{1/2}),   
\end{align*}
with the $\tilde{f_i}$ being unimodular. Now, by the unimodular case (with $C_0\eta^{1/2}$ in place of $\eta$ for some large constant $C_0$), we deduce that 
\begin{align}\label{e141}
\frac{1}{|\mathcal{I}|}\sum_{I\in \mathcal{I}}\sup_{\substack{P(Y)\in \mathbb{Z}[Y]\\\deg{P}\leq k-1}}|\mathbb{E}_{n\in I}\tilde{f_1}(n)e(-P(n)/N)|\geq 1-2\eta^{\gamma(k)}.    
\end{align}
for some partition $\mathcal{I}$ of $\mathbb{Z}_N$ into intervals $I\subset \mathbb{Z}_N$ of length $\in [c_{\varepsilon,C_0\eta^{1/2}}N/2,c_{\varepsilon,C_0\eta^{1/2}}N]$. But by the decomposition of $\tilde{f_1}$,~\eqref{e141} implies~\eqref{e142} (with the value of $\gamma(k)$ in the latter adjusted). 

We may henceforth assume that the $f_j$ are all unimodular. Let us therefore write 
\begin{align}\label{e183}
f_j(n)= e(F_j(n)),\quad F_j\colon  \mathbb{Z}_N\to \mathbb{R}/\mathbb{Z}.   
\end{align}
Let us also note that, for some unimodular $\sigma(d)$ we may write~\eqref{e143} as
\begin{align}\label{e143a}
\mathbb{E}_{d\in \mathbb{Z}_N, 0\leq d\leq \varepsilon N}\sigma(d)\mathbb{E}_{n\in \mathbb{Z}_N}f_1(n+dh_1)\cdots f_k(n+dh_k)\geq 1-\eta. \end{align}
The rest of the proof proceeds in two steps: handling the case $\sigma(d)=1$ and generalizing to arbitrary $\sigma(d)$.\\

\textbf{The case $\sigma(d)\equiv 1$.} Suppose that~\eqref{e143a} holds with $\sigma(d)\equiv 1$. In this case, we will show the slightly stronger conclusion
\begin{align}\label{e141a}
\frac{1}{|\mathcal{I}|}\sum_{I\in \mathcal{I}}\sup_{\substack{P(Y)\in \mathbb{Z}[Y]\\\deg{P}\leq k-2}}|\mathbb{E}_{n\in I}f_1(n)e(-P(n)/N)|\geq 1-2\eta^{\gamma(k-1)}.    
\end{align}
where $k$ has been reduced by one. 

Substituting~\eqref{e183} to~\eqref{e143a} and applying Markov's inequality and the Taylor approximation of the exponential function, the assumption~\eqref{e143a} implies that
\begin{align}\label{e152}
 \mathbb{E}_{d\in \mathbb{Z}_N, 0\leq d\leq \varepsilon N}\mathbb{E}_{n\in \mathbb{Z}_N}\|F_1(n+dh_1)+\cdots +F_k(n+dh_k)-\alpha\|\leq 10\eta^{1/2}.    \end{align}
for some $\alpha \in \mathbb{R}.$

We then make the change of variables
\begin{align*}
n=\sum_{i\leq k}h_ix_i,\quad d=-(x_1+\cdots +x_k).
\end{align*}
Since $N$ is a large prime, we easily see that every pair $(n,d)$ corresponds to essentially the same number $(\varepsilon+o_{N\to \infty}(1))N^{k-2}$ of choices of $(x_1,\ldots, x_k)$. Consequently,~\eqref{e152} gives
\begin{align}\label{e144}
 \mathbb{E}_{(x_1,\ldots, x_k)\in \mathbb{Z}_N^k,0\leq L_0(x_1,\ldots, x_k)\leq \varepsilon N}\|F_1(L_1(x_1,\ldots, x_k))+\cdots +F_k(L_k(x_1,\ldots, x_k))-\alpha\|\ll \eta^{1/2},   
\end{align}
where
\begin{align*}
L_0(\bm{x})=-(x_1+\cdots+x_k),\quad L_j(\bm{x})=\sum_{i\leq k}(h_i-h_j)x_i. \end{align*}
Note that, importantly, $L_j$ is independent of the $j$th coordinate. 

Now, define the discrete derivatives
\begin{align*}
\partial_{\ell}^{(j)}G(x_1,\ldots x_j, \ldots, x_n)\coloneqq G(x_1,\ldots, x_j+\ell, \ldots, x_n)-G(x_1,\ldots, x_j, \ldots, x_n).    
\end{align*}

For any $k\geq 1$, $\varepsilon>0$, define the $k$-dimensional strip
\begin{align*}
S_{k,\varepsilon}\coloneqq \{(x_1,\ldots, x_k)\in \mathbb{Z}_N^k\colon  \,\,0\leq -(x_1+\cdots +x_k)\leq \varepsilon N\}.    
\end{align*}

Then by~\eqref{e144} and Markov's inequality, we have
\begin{align}\label{e178}
\|F_1(L_1(\bm{x}))+\cdots +F_k(L_k(\bm{x}))-\alpha\|\ll \eta^{1/4},\,\, \textnormal{whenever}\,\, \bm{x}\in S_{k,\varepsilon}\setminus \mathcal{E}    
\end{align}
for some exceptional set $\mathcal{E}\subset \mathbb{Z}_N^k$ satisfying $|\mathcal{E}|\ll \eta^{1/4}|S_{k,\varepsilon}|$. By the union bound and~\eqref{e178}, for any $\bm{\ell}\in ([-\varepsilon\eta N,\varepsilon\eta N]\cap \mathbb{Z})^k$ we have 
\begin{align}\label{e145}
\Big\|\sum_{j=1}^k F_j(L_j(\bm{x}+v_1\ell_1+\cdots +v_k\ell_k))-\alpha\Big\|\ll \eta^{1/4}\,\, \textnormal{for all}\,\, v_i\in \{0,1\}^k\,\, \textnormal{whenever}\,\, \bm{x}\in S_{k,\varepsilon}\setminus \mathcal{E}_{\bm{\ell}},    
\end{align}
where $\mathcal{E}_{\bm{\ell}}$ satisfies $|\mathcal{E}_{\bm{\ell}}|\ll \eta^{1/4}|S_{k,\varepsilon}|$. Hence, utilizing ~\eqref{e145}, the triangle inequality for $\|\cdot\|$, and the fact that $L_j(x_1,\ldots, x_k)$ is independent of $x_j$, we see that 
\begin{align}\label{e147}
  \mathbb{E}_{(x_1,\ldots, x_k)\in \mathbb{Z}_N^k,0\leq L_0(x_1,\ldots, x_k)\leq \varepsilon N}\|\partial_{\ell_2}^{(2)}\cdots \partial_{\ell_k}^{(k)}F_1(L_1(0,x_2,\ldots, x_k))\|\ll \eta^{1/4}      
\end{align}
for any integers $|\ell_i|\leq \varepsilon \eta N$. Note that by Markov's inequality if we split $\mathbb{Z}_N$ into intervals of length $\in [\varepsilon N/2, \varepsilon N]$, then all but $\ll \eta^{1/8}$ proportion of these intervals contain an element $x_0$ such that 
\begin{align*}
  \mathbb{E}_{(x_2,\ldots, x_k)\in \mathbb{Z}_N^{k-1},x_0\leq L_0(0,x_2,\ldots, x_k)\leq \varepsilon N+x_0}\|\partial_{\ell_2}^{(2)}\cdots \partial_{\ell_k}^{(k)}F_1(L_1(0,x_2,\ldots, x_k))\|\ll \eta^{1/8}.      
\end{align*}
Hence, summing over different values of $x_0$, we can remove the constraint $x_0\leq L_0(0,x_2,\ldots, x_k)\leq \varepsilon N+x_0$ and obtain
\begin{align}\label{e153}
\mathbb{E}_{(x_2,\ldots, x_k)\in \mathbb{Z}_N^{k-1}}\|\partial_{\ell_2}^{(2)}\cdots \partial_{\ell_k}^{(k)}F_1(L_1(0,x_2,\ldots, x_k))\|\ll \eta^{1/8}.        
\end{align}

We now need the following lemma that allows us to ``integrate'' this conclusion.

\begin{lemma}[Derivative of $G$ being small in $L^1$ implies that $G$ is almost constant]\label{le_integrate}
Let $\eta>\varepsilon>0$. Let $I\subset \mathbb{Z}_N$ be an interval of length $\geq \varepsilon\eta N$. Suppose that a function $G\colon  I\to \mathbb{R}/\mathbb{Z}$ satisfies
\begin{align}\label{e146}
\mathbb{E}_{x\in I}\|\partial_{\ell}G(x)\|&\ll \eta\quad \textnormal{for all}\quad \ell\in [-\varepsilon\eta N,\varepsilon\eta N].
\end{align}
Then, given any partition $\mathcal{I}$ of $I$ into intervals of lengths $\in [\varepsilon\eta N/200,\varepsilon\eta N/100]$, there exist constants $c_J\in  \mathbb{R}/\mathbb{Z}$ such that 
\begin{align*}
\mathbb{E}_{J\in \mathcal{I}}\mathbb{E}_{x\in J}\|G(x)-c_J\|&\ll \eta^{1/2}.
\end{align*}
\end{lemma}

\begin{proof}
By~\eqref{e146} and Markov's inequality, for each $\ell\in [- \varepsilon\eta N,\varepsilon\eta N]$ there exists a set $\mathcal{E}'_{\ell}\subset I$ with $|\mathcal{E}'_{\ell}|\ll \eta^{1/2}|I|$ such that
\begin{align*}
 \|G(x+\ell)-G(x)\|\ll \eta^{1/2}\quad \textnormal{whenever}\quad x\not \in \mathcal{E}'_{\ell}.    
\end{align*}
Let $\mathcal{X}$ be the set of $x\in I$ that belong to at most $\varepsilon \eta N/10$ of the sets $\mathcal{E}_{\ell}'$. Note that 
$$2\varepsilon\eta  N\cdot \eta^{1/2}|I|\geq \sum_{\ell}\sum_{x\in I\setminus \mathcal{X}}1_{x\in \mathcal{E}_{\ell}'}=\sum_{x\in I\setminus \mathcal{X}}\sum_{\ell}1_{x\in \mathcal{E}_{\ell}'}\geq |I\setminus \mathcal{X}|\cdot \varepsilon\eta N/10,$$
so $|I\setminus \mathcal{X}|\ll \eta^{1/2} |I|$.

Now, for any $x,y\in \mathcal{X}$ with $|x-y|\leq \varepsilon\eta N/100$, we can find $\ell\in [-\varepsilon\eta N,\varepsilon\eta N]$ such that $x\not \in \mathcal{E}'_{\ell}$ and $y\not \in \mathcal{E}'_{\ell+x-y}$. This means that
\begin{align*}
\|G(x+\ell)-G(x)\|\ll \eta^{1/2}\quad \textnormal{and}\quad \|G(x+\ell)-G(y)\|\ll \eta^{1/2},    
\end{align*}
so by the triangle inequality for all $x,y\in \mathcal{X}$ with $|x-y|\leq \varepsilon\eta N/100$ we have
\begin{align*}
\|G(x)-G(y)\|\ll \eta^{1/2}.    
\end{align*}
Now, letting $\mathcal{I}$ be any partition of $I$ into intervals whose lengths are in $[\varepsilon\eta N/200,\varepsilon\eta N/100]$, we deduce that for some constants $c_J$ we have 
\begin{align*}
\mathbb{E}_{J\in \mathcal{I}}\mathbb{E}_{x\in J}\|G(x)-c_J\|\ll \eta^{1/2},    
\end{align*}
which was the claim.
\end{proof}

Now we shall apply Lemma~\ref{le_integrate} to finish the proof of the case $\sigma(d)\equiv 1$. By~\eqref{e153} and Markov's inequality, for proportion $\geq 1-\eta^{1/16}$ of $(x_3,\ldots, x_k)\in \mathbb{Z}_N^{k-2}$ we have
\begin{align}\label{e185b}
 \mathbb{E}_{x_2\in \mathbb{Z}_N}\|\partial_{\ell_2}^{(2)}\cdots \partial_{\ell_k}^{(k)}F_1(L_1(0,x_2,\ldots, x_k))\|\ll \eta^{1/16}.    
\end{align}
Applying Lemma~\ref{le_integrate} to this, there exists a partition $\mathcal{I}_2$ of $\mathbb{Z}_N$ (which is independent of $x_3,\ldots, x_k$) such that 
\begin{align}\label{e185}
\mathbb{E}_{J_2\in \mathcal{I}_2}\mathbb{E}_{(x_2,\ldots, x_k)\in J_2\times \mathbb{Z}_N^{k-2}}\|\partial_{\ell_3}^{(3)}\cdots \partial_{\ell_k}^{(k)}F_1(L_1(0,x_2,\ldots, x_k))-c_{J_2}(x_3,\ldots, x_k)\|\ll \eta^{1/32}   
\end{align}
for some functions $c_{J_2}(x_3,\ldots, x_k)$. We wish to show that $c_{J_2}(x_3,\ldots, x_k)$ can in fact be taken to be independent of $x_3,\ldots, x_k$. To this end, note that by Markov's inequality for proportion $\geq 1-\eta^{1/64}$ of $y_0\in \mathbb{Z}_N$ and $J_2\in \mathcal{I}_2$ we have
\begin{align*}
\mathbb{E}_{(x_2,\ldots, x_k)\in J_2\times \mathbb{Z}_N^{k-2}, L_1(0,x_2,\ldots, x_k)=y_0}\|\partial_{\ell_3}^{(3)}\cdots \partial_{\ell_k}^{(k)}F_1(y_0)-c_{J_2}(x_3,\ldots, x_k)\|\ll \eta^{1/64}.    
\end{align*}
Averaging over all such $y_0$ and evaluating the $x_2$ average first, we see that for each $J_2\in \mathcal{I}_2$ there exists a constant $c_{J_2}\in \frac{k!}{N}\mathbb{Z}$ such that
\begin{align}\label{e186}
\mathbb{E}_{J_2\in \mathcal{I}_2}\mathbb{E}_{(x_3,\ldots, x_k)\in \mathbb{Z}_N^{k-2}}\|c_{J_2}-c_{J_2}(x_3,\ldots, x_k)\|\ll \eta^{1/64}.  
\end{align}
Hence, we may indeed replace $c_{J_2}(x_3,\ldots, x_k)$ with a constant $c_{J_2}$ in~\eqref{e185}.

Next, note that for some constant $\gamma_{J_2}\in \frac{1}{N}\mathbb{Z}$ we have
\begin{align*}
c_{J_2}=\partial_{\ell_3}^{(3)}\cdots \partial_{\ell_k}^{(k)}\gamma_{J_2}(\sum_{j\leq k}(h_j-h_1)x_j)^{k-2},    
\end{align*}
and hence by~\eqref{e185} and~\eqref{e186} there exist polynomials $P_{k-2,J_2}\in \frac{1}{N}\mathbb{Z}[Y]$ of degree $\leq k-2$ such that
\begin{align*}
\mathbb{E}_{J_2\in \mathcal{I}_2}\mathbb{E}_{(x_2,\ldots, x_k)\in J_2\times \mathbb{Z}_N^{k-2}}\|\partial_{\ell_3}^{(3)}\cdots \partial_{\ell_k}^{(k)}(F_1-P_{k-2,J_2}/N)(L_1(0,x_2,\ldots, x_k))\|\ll \eta^{1/32}.    
\end{align*}
But now we may repeat  for the function $F_1-P_{k-2,J_2}/N$ all the previous arguments (that is, applying Markov's inequality to find many $(x_4,\ldots ,x_k)$ for which the average over $x_3$ is large, and applying Lemma~\ref{le_integrate}) to obtain
\begin{align*}
\mathbb{E}_{J_2\in \mathcal{I}_2,J_3\in \mathcal{I}_3}\mathbb{E}_{(x_2,\ldots, x_k)\in J_2\times J_3\times \mathbb{Z}_N^{k-3}}\|\partial_{\ell_4}^{(4)}\cdots \partial_{\ell_k}^{(k)}(F_1-P_{k-2,J_2,J_3}/N)(L_1(0,x_2,\ldots, x_k))\|\ll \eta^{1/32^2}    
\end{align*}
for some polynomial $P_{k-2,J_2,J_3}\in \frac{1}{N}\mathbb{Z}[Y]$ of degree $\leq k-2$. Continuing inductively, we see that there exist polynomials $P_{J_2,\ldots, J_k}\in \frac{1}{N}\mathbb{Z}[Y]$ of degree $\leq k-2$ such that 
\begin{align*}
\mathbb{E}_{J_i\in \mathcal{I}_i\,\, \forall 2\leq i\leq k}\mathbb{E}_{(x_2,\ldots, x_k)\in J_2\times \cdots \times J_k}\|(F_1-P_{J_2,\ldots, J_k}/N)(L_1(0,x_2,\ldots, x_k))\|\ll \eta^{1/64^k}. \end{align*}
By the pigeonhole principle applied to the $x_3,\ldots, x_k$ variables, we see for some choice of $J_3,\ldots, J_k$ that
\begin{align*}
 \mathbb{E}_{J_2\in \mathcal{I}_2}\mathbb{E}_{x\in J_2}\|(F_1-\frac{1}{N}P_{J_2,\ldots, J_k})((h_2-h_1)x+C(x_3,\ldots,x_k))\|\ll \eta^{1/64^k}   
\end{align*}
for $(x_3,\ldots, x_k)\in J_3\times \cdots \times J_k$ and some $C(x_3,\ldots, x_k)\in \mathbb{Z}_N$. Finally, by Taylor approximation of the exponential function, this gives
\begin{align}\label{e149}
 \mathbb{E}_{J_2\in \mathcal{I}_2}  \mathbb{E}_{x\in J_2}|f_1(x)e(-\tilde{P}_{J_2}(x)/N)-1|\ll \eta^{1/64^{k+1}}   
\end{align}
for some polynomials $\tilde{P}_{J_2}\in \mathbb{Z}[Y]$ of degree $\leq k-2$. Since $\mathcal{I}_2$ is a partition, the claim is now proved with $\gamma(k-1)=64^{-k-2}$, say.\\

\textbf{The case of arbitrary $\sigma(d)$.} Suppose that~\eqref{e143a} holds. By Cauchy--Schwarz, we have
\begin{align}\label{e187}
\mathbb{E}_{d\in \mathbb{Z}_N,0\leq d\leq \varepsilon N}\Big|\mathbb{E}_{n\leq x}\prod_{j=1}^k f_j(n+dh_j)\Big|^2\geq 1-10\eta.
\end{align}
Let $L=\eta^2 N$. We apply van der Corput's inequality (see e.g.~\cite[Lemma 4.1]{gt-nilmanifold}) to~\eqref{e187} in the form
\begin{align*}
|\mathbb{E}_{n\in \mathbb{Z}_N}a(n)|^2\leq \textnormal{Re}(2\mathbb{E}_{\ell\in \mathbb{Z}_N,1\leq \ell\leq L}(1-\ell/L)\mathbb{E}_{n\in \mathbb{Z}_N}a(n)\overline{a(n+\ell)})+O(1/L+L/N). \end{align*}
Substituting this into~\eqref{e187} and exchanging orders of summation, we deduce that
\begin{align*}
2\mathbb{E}_{\ell\in \mathbb{Z}_N,1\leq \ell\leq L}(1-\ell/L)\Big|\mathbb{E}_{d\in \mathbb{Z}_N,0\leq d\leq \varepsilon N}\mathbb{E}_{n\leq x}\prod_{j=1}^k \partial_{\ell}f_j(n+dh_j)\Big|\geq 1-100\eta, 
\end{align*}
where $\partial_{\ell}f(n)\coloneqq f(n+\ell)\overline{f}(n)$.
By Markov's inequality, this gives
\begin{align*}
\Big|\mathbb{E}_{d\in \mathbb{Z}_N,0\leq d\leq \varepsilon N}\mathbb{E}_{n\leq x}\prod_{j=1}^k \partial_{\ell}f_j(n+dh_j)\Big|\geq 1-100\eta^{1/2}     
\end{align*}
for proposition $\geq 1-100\eta^{1/2}$ of $0\leq \ell\leq L$. Hence, by the case $\sigma(d)\equiv 1$ that was already handled, there exist polynomials $P_{\ell,I}$ of degree $\leq k-2$ such that
\begin{align*}
\frac{1}{|\mathcal{I}|}\sum_{I\in \mathcal{I}}|\mathbb{E}_{n\in I}\partial_{\ell}f_1(n)\cdot e(-P_{\ell,I}(n)/N))|\geq 1-2\eta^{\gamma(k-1)}   
\end{align*}
for proportion $\geq 1-100\eta^{1/2}$ of $0\leq \ell\leq L$ (where the partition $\mathcal{I}$ is independent of $\ell$). Now, applying van der Corput's inequality $k-2$ times, we see that
\begin{align*}
\frac{1}{|\mathcal{I}|}\sum_{I\in \mathcal{I}}|\mathbb{E}_{n\in I}\partial_{\ell_1}\cdots \partial_{\ell_{k-1}} f_1(n)|\geq 1-C_k\eta^{\gamma(k-1)/10^{k-2}}   
\end{align*}
for proportion $\geq (1-100\eta^{1/2})^{k-1}$ of all $(\ell_1,\ldots, \ell_{k-1})$ with $1\leq \ell_i\leq L$. Hence, writing $f_1(n)=e(F_1(n))$, we have 
\begin{align*}
\frac{1}{|\mathcal{I}|}\sum_{I\in \mathcal{I}}\mathbb{E}_{n\in I}\|\partial_{\ell_1}\cdots \partial_{\ell_{k-1}}F_1(n)-\alpha\|\ll_k \eta^{\gamma(k-1)/10^{k-2}}   
\end{align*}
for some $\alpha\in \mathbb{R}$ and the same tuples $(\ell_1,\ldots, \ell_{k-1})$ as above.

But now the rest of the argument proceeds as in the $\sigma(d)\equiv 1$ case from~\eqref{e185b} onward (adjusting the value of $\gamma(k)$).
\end{proof}

\subsection{A Fourier analysis argument}

We shall then prove Proposition~\ref{prop_polyphase}; this is the fourth and final step in  the proof of Theorem~\ref{theo_multiplicative-nonpret}. It suffices to show that if
\begin{align*}
X\geq X_0(H),\quad 1\leq Q\ll \log \log H,    
\end{align*}
for suitably chosen $X_0(\cdot)$, then
\begin{align}\label{e140}
\sup_{\substack{P(Y)\in \mathbb{R}[Y]\\\deg{P}\leq k}}\sup_{1\leq a\leq Q}|\mathbb{E}_{x\leq n\leq x+H} g(Qn+a)e(-P(n))|\leq 1-\delta
\end{align}
for all $x\leq X$ outside an exceptional set\footnote{The exceptional set bound $O(X/(\log \log H)^2)$ could be improved, but for our purposes any expression of the form $o_{H\to \infty}(X)$ suffices.} of measure $\leq X/(\log \log H)^2$, where $\delta>0$ depends on $q$ only. Indeed, once we have this, we may apply the union bound to deduce that the average over $n\in [x,x+H]$ in~\eqref{e140} can be replaced with the supremum of averages over $n\in I$, where $I\subset [x,x+H]$ has length $\geq H/Q$.  

For each $x\leq X$, let $P_x$ be a polynomial of degree $\leq k$ (and without loss of generality having constant coefficient zero) that maximizes the left-hand side of~\eqref{e140}. Such a polynomial exists by compactness. 

Let us split the set of $x\leq X$ according to whether the polynomial $P_x$ is major or minor arc. Let 
$$L=\log \log H,$$
and define
\begin{align*}
\mathcal{X}_1&\coloneqq \{x\leq X\colon   P_x(n)=\sum_{1\leq j\leq k}c_jn^j\colon  \quad \forall\, j\leq k\colon  \,\, \left|c_j-\frac{a_j}{q_j}\right|\leq \frac{1}{q_j(H/L)^j}\,\, \textnormal{for some}\,\, |a_j|,q_j\leq L\},\\
\mathcal{X}_2&\coloneqq [0,X]\setminus \mathcal{X}_1.
\end{align*}
\textbf{Case 1: Major arcs.} To prove~\eqref{e140} in the major arc case $x\in \mathcal{X}_1$, we start with the following lemma. 

\begin{lemma}\label{le_pret}
 Let $q\in \mathbb{N}$ be fixed,  and let $g\colon  \mathbb{N}\to \mu_q$ be multiplicative. Let $r\in \mathbb{N}$, and let $\chi\pmod r$ be a Dirichlet character. Then, for all $x\geq 2$, we have
 \begin{align*}
 \inf_{|t|\leq x}\mathbb{D}(g,\chi(n)n^{it};x)\geq \frac{1}{3q}\mathbb{D}(g,\chi(n);x)-O(\log r).
 \end{align*}
\end{lemma}

\begin{proof}
 This is similar to~\cite[Lemma C.1]{mrt-average}. We may assume that $r\leq \log x$. Let $t$ satisfy $|t|\leq x$. By the triangle inequality for the pretentious distance and the fact that $g^{q}=1$, we have
 \begin{align}\label{eq:pret1}
 q\mathbb{D}(g,\chi(n)n^{it};x)\geq \mathbb{D}(1,\chi^{q}(n)n^{iqt};x).    
 \end{align}
By the Vinogradov--Korobov zero-free region, if either $|t|\geq 1$ or $\chi^q$ is non-principal, this is 
\begin{align*}
\geq \sqrt{\left(\frac{1}{3}-o(1)\right)\log \log x} \geq \left(\frac{1}{\sqrt{6}}-o(1)\right)\mathbb{D}(g,\chi;x),
\end{align*}
since necessarily $\mathbb{D}(g,\chi;x)\leq \sqrt{2\log \log x+O(1)}$ by Mertens's theorem. We may therefore assume that $|t|\leq 1$ and that $\chi^{q}=\chi_0$ is principal. 

Since $\mathbb{D}(1,\chi_0;x)=O(\log r)$, returning to~\eqref{eq:pret1} we obtain the desired claim unless
\begin{align}\label{eq:pret2}
 \mathbb{D}(1,n^{iqt};x)\leq \frac{1}{3}\mathbb{D}(g,\chi;x).   
\end{align}
From the prime number theorem, we see for $|u|\leq q$ that $\mathbb{D}(1,n^{iu};x)=\log(1+|u|\log x)+O(1)$. Hence, $\mathbb{D}(1,n^{iqt};x)=\mathbb{D}(1,n^{it};x)+O(1)$, so~\eqref{eq:pret2} gives
\begin{align}\label{eq:pret3}
 \mathbb{D}(1,n^{it};x)\leq \frac{1}{3}\mathbb{D}(g,\chi;x)+O(1).   
\end{align}
But~\eqref{eq:pret3} and the triangle inequality for the pretentious distance imply
\begin{align*}
 \mathbb{D}(g,\chi(n)n^{it};x)\geq \mathbb{D}(g\overline{\chi},1;x)-\mathbb{D}(1,n^{it};x)\geq \frac{2}{3}\mathbb{D}(g,\chi;x).   
\end{align*}
The claim now follows in all cases. 
\end{proof}

In the major arc case, we will actually prove~\eqref{e140} with $(\log \log H)^{-1}$ in place of $1-\delta$. For $x\in \mathcal{X}_1$, we may write
\begin{align*}
P_x(n)=P_x^{(1)}(n)+P_x^{(2)}(n),\quad P_x^{(1)}(n)\coloneqq \sum_{1\leq j\leq k}\frac{a_j}{q_j}n^{j}.    
\end{align*}
Thus, by partial summation,
\begin{align*}
\left|\mathbb{E}_{x\leq n\leq x+H}  g(Qn+a)e(-P_x(n))\right|\ll L^k \sup_{\substack{I\subset [x,x+H]\\|I|\geq H/L^{2k}}} \left|\mathbb{E}_{n\in I}g(Qn+a)e(-P_x^{(1)}(n))\right|+1/L^k.
\end{align*}
By expanding $e(-P_x^{(1)}(n))$ in terms of Dirichlet characters modulo $R\coloneqq q_1\cdots q_k$, we find
\begin{align}\label{eq21b}
\sup_{\substack{I\subset [x,x+H]\\|I|\geq H/L^{2k}}}\left|\mathbb{E}_{n\in I}  g(Qn+a)e(-P_x^{(1)}(n))\right|\ll R \sup_{\chi\pmod{R}}\,\sup_{m\mid R}\,\sup_{\substack{I\subset [x,x+H]\\|I|\geq H/L^{2k}}}|\mathbb{E}_{n\in I}  g(Qn+a)\chi(n/m)1_{m\mid n}|.    \end{align}
Writing $n=m(Rn'+b)$ for some $1\leq b\leq R$, and expanding out $1_{n\equiv a+Qmb \pmod{QRm}}$ in terms of Dirichlet characters, we can bound the previous expression by
\begin{align}\label{e150}
\ll R^2Q\sup_{\xi\pmod{R^2Q}}\,\sup_{m'\mid R^2Q}\,\sup_{\substack{I'\subset [x/(Qm'),(x+H)/(Qm')]\\|I'|\geq H/(R^2QL^{2k})}}|\mathbb{E}_{n\in I'} g(m'n)\xi(n)|. 
\end{align}
Recall that we are working under the non-pretentiousness assumption $\mathbb{D}(g,\xi;\infty)=\infty$. Since $g$ takes values in $\mu_q$, we can boost this assumption to 
\begin{align}\label{e186a}
\inf_{|t|\leq x}\mathbb{D}(g,\xi(n)n^{it};x)\xrightarrow{x\to \infty} \infty     
\end{align}
for all $\xi\pmod{R^2Q}$ by Lemma~\ref{le_pret}. 

By~\eqref{e186a} and the Matom\"aki--Radziwi\l{}\l{} theorem, in the form of~\cite[Theorem A.1]{mrt-average}, for $X$ large enough in terms of $H'$ we conclude that
\begin{align*}
|\mathbb{E}_{y\leq n\leq y+H'}g(m'n)\xi(n)|\ll (\log \log H')/\log H'     
\end{align*}
for all natural numbers $y\leq X$, apart from $\ll X(\log \log H')/\log H'$ exceptions.

By the union bound (which allows us to restrict to fixed $\xi,m',I'$ in~\eqref{e150}), we now see that~\eqref{e150} is $\ll (\log H)^{-1/2}$ for $x\in \mathcal{X}_1\setminus \mathcal{E}$, where $\mathcal{E}\subset \mathbb{N}$ is some exceptional set with $|\mathcal{E}\cap [1,X]|\ll X/(\log H)^{1/2}$.\\

\textbf{Case 2: Minor arcs.}
We are now left with the minor arc case $x\in \mathcal{X}_2$. We begin with two lemmas, the first of which says that $P_x(n)\pmod 1$ is uniformly distributed for $x\in \mathcal{X}_2$. In what follows, let
\begin{align*}
\eta=L^{-1}=(\log \log H)^{-1}.    
\end{align*}

\begin{lemma}[Erd\H{o}s--Tur\'an]\label{le_equidist}
Let $x\in \mathcal{X}_2$. There exists a constant $C=C_k\geq 1$ such that for any arc $I$ of the unit circle of $\mathbb{C}$ we have
\begin{align*}
\mathbb{E}_{x\leq n\leq x+H}1_{e(P_x(n))\in I}=|I|+O(\eta^{1/C}).   
\end{align*}
\end{lemma}

\begin{proof}
By the Erd\H{o}s--Tur\'an inequality~\cite[Chapter 1]{montgomery}, for any $1\leq M_0\leq H$ we have
\begin{align}\label{e151}
|\mathbb{E}_{x\leq n\leq x+H}1_{e(P_x(n))\in I}-|I||\ll \frac{1}{M_0}+\sum_{1\leq j\leq M_0}\frac{1}{j}|\mathbb{E}_{x\leq n\leq x+H}e(jP_x(n))|.    
\end{align}
Take $M_0=\eta^{-1/C}$, where  $C=C_k$ is large. Then, by the assumption $x\in \mathcal{X}_2$ and standard estimates for Weyl sums (see~\cite[Lemma 4.4]{gt-nilmanifold}), we get a bound of $\ll \eta^{1/C}$ for the right-hand side of~\eqref{e151}. This completes the proof.
\end{proof}

\begin{lemma}[Maximizing an exponential sum]\label{le_inequality}
Let $n\geq 1$ be an integer, let $m\in [1,n]$ be a real number, and let $w_j\in [0,1]$ be real numbers for $1\leq j\leq n$ with $\sum_{1\leq j\leq n}w_j=m$. Then 
\begin{align}\label{eq:00}
\left|\sum_{1\leq j\leq n}w_je\left(\frac{j}{n}\right)\right|\leq \left|\sum_{1\leq j\leq \lfloor m\rfloor}e\left(\frac{j}{n}\right)\right|+\{m\}.   
\end{align}
If equality holds, then after a cyclic permutation of the indices we have $w_j=1$ for $j\leq m$ and $w_j=0$ for $\lceil m \rceil <j\leq n$.
\end{lemma}

\begin{proof}
Consider first the case where $m$ is an integer.

Let $\theta$ be a real number such that
\begin{align*}
e(\theta) \sum_{1\leq j\leq n}w_je\left(\frac{j}{n}\right)\geq 0.   
\end{align*}
Then 
\begin{align*}
\left|\sum_{1\leq j\leq n}w_je\left(\frac{j}{n}\right)\right|=\sum_{1\leq j\leq n}w_j\textnormal{Re}\left(e\left(\frac{j}{n}+\theta\right)\right).
\end{align*}
The summand above equals to $w_j\cos(2\pi(\frac{j}{n}+\theta))$. Since the $m$ largest values of $\cos(2\pi(\frac{j}{n}+\theta))$, with $j\in \mathbb{Z}_n$, are attained for $j\in I$, where $I$ some interval of $\mathbb{Z}_n$ of length $m$, the claim follows from the rearrangement inequality.

Now, suppose that $m$ is not an integer. Let $k$ be the least integer for which $$\sum_{1\leq j\leq k}w_j\geq \{m\}.$$
Let $w_j'=0$ for $j\leq k-1$, $w_k'=w_k-\{m\}+\sum_{1\leq j\leq k-1}w_j$, and $w_j'=w_j$ for $j>k$. Then $0\leq w_j'\leq 1$ and $\sum_{1\leq j\leq n}w_j'=\lfloor m\rfloor$, and by the triangle inequality $$\left|\sum_{j\leq n}w_je\left(\frac{j}{n}\right)\right|\leq \sum_{j\leq k-1}w_j+\left|\sum_{k\leq j\leq n}w_je\left(\frac{j}{n}\right)\right|\leq  \left|\sum_{j\leq n}w_j'e\left(\frac{j}{n}\right)\right|+\{m\}.$$ 
Now the claim follows from the case where $m$ is an integer.
\end{proof}

Now, let $J_0=\lfloor \eta^{-1/(2C)}\rfloor$. Recall that $g^{q}\equiv 1$. Thus, we have
\begin{align*}
&\mathbb{E}_{x\leq n\leq x+H}g(Qn+a)e(P_x(n))\\
&\leq \sum_{0\leq b<q} \left|\sum_{0\leq j<J_0}e\left(\frac{j}{J_0}\right)\mathbb{E}_{x\leq n\leq x+H}1_{g(Qn+a)=e(b/q)}1_{P_x(n)\in [j/J_0,(j+1)/J_0]\pmod 1}\right|+O(\eta^{1/(2C)}).
\end{align*}
Denote
\begin{align*}
\alpha(j)\coloneqq  \mathbb{E}_{x\leq n\leq x+H}1_{g(Qn+a)=e(b/q)}1_{P_x(n)\in [j/J_0,(j+1)/J_0]\pmod 1}\in [0,1].  
\end{align*}
Then 
\begin{align}\label{eq22}
\sum_{0\leq j<J_0}\alpha(j)=\delta_b,\quad \textnormal{where}\quad  \delta_b\coloneqq \mathbb{E}_{x\leq n\leq x+H}1_{g(Qn+a)=e(b/q)}.    
\end{align}
Moreover, by Lemma~\ref{le_equidist}, we have
\begin{align}\label{eq23}
\alpha(j)\leq  \mathbb{E}_{x\leq n\leq x+H}1_{P_x(n)\in [j/J_0,(j+1)/J_0]\pmod 1}\leq \frac{1}{J_0}+O(\eta^{1/C})\leq \frac{1+\eta^{1/(2C)}}{J_0}.    
\end{align}
Combining~\eqref{eq22} and~\eqref{eq23} with Lemma~\ref{le_inequality} (applied to the real numbers $\alpha(j)\frac{J_0}{1+\eta^{1/(2C)}}\in [0,1]$), we obtain
\begin{align*}
\left|\sum_{0\leq j<J_0}e\left(\frac{j}{J_0}\right)\alpha(j)\right|&\leq \frac{1+\eta^{1/(2C)}}{J_0}\left|\sum_{1\leq j\leq \delta_b J_0}e\left(\frac{j}{J_0}\right)\right|+O\left(\frac{1}{J_0}\right)\\
&\leq (1+2\eta^{1/(2C)})\left|\int_{0}^{\delta_b}e(t)\, dt\right|+O(\eta^{1/(2C)})\\
&\leq (1+3\eta^{1/(2C)})\frac{\sin(\pi \delta_b)}{\pi}.    
\end{align*}

The function $f(y)=\sin(\pi y)$ satisfies $f''(y)=-\pi^2\sin(\pi y)\leq 0$ for $y\in [0,1]$, so $f$ is concave. Therefore, by Jensen's inequality,
\begin{align*}
 \frac{1}{\pi}\sum_{0\leq b<q}\sin(\pi \delta_b)\leq \frac{q}{\pi}\sin(\pi\mathbb{E}_{0\leq b<q}\delta_b)=\frac{q}{\pi}\sin\left(\frac{\pi}{q}\right)\coloneqq 1-2\delta   
\end{align*}
for some $\delta>0$ depending only on $q$ (here we used the fact that $\sin x<x$ for $x>0$).

We now conclude that for any $x\in \mathcal{X}_2$ we have 
\begin{align*}
|\mathbb{E}_{x\leq n\leq x+H}g(Qn+a)e(P_x(n))|\leq (1+3\eta^{1/(2C)})(1-2\delta)+O(\eta^{1/(2C)})\leq 1-\delta   
\end{align*}
since $\eta>0$ can be made arbitrarily small by taking $H$ to be large. Recall from the major arc case that for $x\in \mathcal{X}_1$ we had
\begin{align*}
\sup_{1\leq a\leq Q}\left|\mathbb{E}_{x\leq n\leq x+H}g(Qn+a)e(P_x(n))\right|\leq \eta,   
\end{align*}
unless $x\leq X$ belongs to an exceptional set of size $\ll X/\log \log H$. Therefore, for all $x\leq X$ outside an exceptional set of size $\ll X/\log \log H$ we have 
\begin{align*}
\sup_{1\leq a\leq Q}|\mathbb{E}_{x\leq n\leq x+H}g(Qn+a)e(P_x(n))|\leq 1-\delta.  
\end{align*}
 As noted earlier, this was enough to complete the proof of Proposition~\ref{prop_polyphase}.

\section{The pretentious case}\label{sec:pret}

We will next prove our result on pretentious functions, Theorem~\ref{theo_multiplicative-pret}. 

We remark that there is an explicit asymptotic formula for $\sum_{n\leq x}g(P(n))$ due to Klurman~\cite{klurman}, but we were not able to utilize it here, since the error term $\mathbb{D}_P(g,1;\log x,x)$ in that formula (as defined in~\cite{klurman}) is guaranteed to be small only in the case where $P(n)$ is almost always $O(n^{1+o(1)})$-smooth, which should only happen when $P(x)$ is linear. Of course, we could first factorize $g=g_sg_{l}$, where $g_s$ is the ``smooth part'' of $g$ and $g_l$ is the ``rough part'', and apply Klurman's formulas to show that $g_s(P(n))$ exhibits some oscillations, but unfortunately it does not seem possible to conclude from this that $g(P(n))$ itself is oscillating when \emph{restricted} to the set where $g$ and $g_s$ agree (even though we know by assumption that this set is large).  

The proof of Theorem~\ref{theo_multiplicative-pret} is instead based on the principle (which also underlies the proofs in~\cite{klurman}) that if $g$ is pretentious, then the ``$W$-tricked'' functions $n\mapsto g(Wn+n_0)$ become ``almost periodic'' in a suitable sense as $w$ tends to infinity (if $t=0$; if instead $t\neq 0$, then these functions behave like Archimedean characters). We formalize this principle as the following lemma.

\begin{lemma}[Almost periodicity of $W$-tricked pretentious functions]\label{le_concentration} Let $x\geq 10$, let $2\leq w\leq \log x$ be an integer, and let $g\colon  \mathbb{N}\to \mathbb{D}$ be multiplicative. Suppose that there exist a Dirichlet character $\chi$ and a real number $t$ such that $g(p)=\chi(p)p^{it}$ for all $p>x$. Denote $W=\prod_{p\leq w}p^{w}$, let $P(x) \in \mathbb{Z}[x]$ be any irreducible polynomial of degree $d\geq 1$ with leading coefficient $c\geq 1$, and let $0\leq b\leq W$ be any integer  with $P(b)>0$ and $p^{\lfloor w/2\rfloor}\nmid P(b)$ for all $p\leq w$. Then we have 
\begin{align*}
&\frac{1}{x}\sum_{n\leq x}|g(P(Wn+b))-(c(Wn)^d)^{it}\widetilde{\chi}(P(b))\exp(F(W))g'_{\leq w}(P(b))|\\
&\ll_{P,\chi,t} \mathbb{D}(g,\chi(n)n^{it};w,\infty)+w^{-1/3},    
\end{align*}
where $g'_{\leq w}\colon  \mathbb{N}\to \mathbb{D}$ is the completely multiplicative function given by $g'_{\leq w}(p)=g(p)\overline{\widetilde{\chi}(p)}p^{-it}$ for $p\leq w$ and $g'_{\leq w}(p)=1$ for $p>w$, and $\widetilde{\chi}$ is a modified character given on the primes by $\widetilde{\chi}(p)=\chi(p)$ if $p\nmid \textnormal{cond}(\chi)$ and $\widetilde{\chi}(p)=1$ otherwise, and $F(W)\coloneqq \sum_{w<p\leq x}\rho(p)\frac{g(p)\overline{\chi(p)}p^{-it}-1}{p}$, where $\rho(q)$ denotes the number of solutions to the congruence $P(y)\equiv 0\pmod{q}$.   
\end{lemma}

\begin{proof}
This is similar to~\cite[Lemma 2.5]{kmpt}; for the sake of completeness we give details. 

We may assume throughout that $P, \chi,t$ are fixed and that $w$ is large enough in terms of them, as otherwise the claim is trivial. Note that by Taylor approximation we have $P(Wn+b)^{it}=(c(Wn)^d)^{it}+O(1/n)$. Also note that as $W\mid P(Wn+b)-P(b)$ and $p^{\lfloor w/2\rfloor}\nmid P(b)$ for all $p\leq w$, we have $\widetilde{\chi}(P(Wn+b))=\widetilde{\chi}(P(b))$. Hence we can reduce to the case $t=0$, $\chi=1$.

Now factorize $g=g_{\leq w}'g_{>w}'$, where $g_{>w}'$ is the completely multiplicative function given on prime powers by $g_{>w}'(p^k)=1$ for $p\leq w$ and $g_{>w}'(p^k)=g(p^k)$ for $p>w$. It suffices to show that we have
\begin{align}\label{eq:turan1}
\frac{1}{x}\sum_{n\leq x}|g_{>w}'(P(Wn+b))-\exp(F(W))|\ll \mathbb{D}(g,1;w,\infty)+w^{-1/3}.   \end{align}

Define an additive function $h$ on prime powers by $h(p^k)\coloneqq g_{>w}(p^k)-1$. Then, using $z=e^{z-1+O(|z-1|^2)}$ for $|z|\leq 1$, the left-hand side of~\eqref{eq:turan1} becomes
\begin{align}\label{eq:turan2}
\frac{1}{x}\sum_{n\leq x}\Big|\prod_{p^k\mid \mid P(Wn+b)}\exp(h(p^k)+O(|h(p^k)|^2))-\exp(F(W))\Big|.  
\end{align}
Note that we have the inequality 
\begin{align}\label{eq:turan0}
\Big|\prod_{j=1}^mz_j-\prod_{j=1}^m w_j\Big|\leq \sum_{j=1}^m|z_j-w_j|\quad \textnormal{  for  }|z_j|, |w_j|\leq 1,
\end{align}
which is easily established by induction on $m$. Using ~\eqref{eq:turan0}, we see that~\eqref{eq:turan2} is
\begin{align}\label{eq:turan3}
\frac{1}{x}\sum_{n\leq x}|\exp(h(P(Wn+b)))-\exp(F(W))|+O\Big(\frac{1}{x}\sum_{n\leq x}\sum_{\substack{p^k\mid \mid P(Wn+b)\\ w<p\leq x}}|h(p^k)|^2\Big).  
\end{align}
From the inequality $|e^{z_1}-e^{z_2}|\ll |z_1-z_2|$ for $\textnormal{Re}(z_1), \textnormal{Re}(z_2)\leq 0$, we see that this is
\begin{align*}
\ll  \frac{1}{x}\sum_{n\leq x}|h(P(Wn+b)))-F(W)|+\frac{1}{x}\sum_{n\leq x}\sum_{\substack{p^k\mid \mid P(Wn+b)\\ w<p\leq x}}|h(p^k)|^2.   
\end{align*}

We can write $h=h_{(w,x^{1/2})}+h_{[x^{1/2},x]}$, where for an interval $I$ the additive function $h_I$ is given on prime powers by $h_I(p^k)=h(p^k)1_{I}(p)$. Also write $F(W)=F_{(w,x^{1/2})}(W)+F_{[x^{1/2},x]}(W)$, where $F_I(W)$ is defined similarly as $F(W)$ but with the sum restricted to $p\in I$. By the triangle inequality, it follows that~\eqref{eq:turan3} is
\begin{align*}
 &\ll \frac{1}{x}\sum_{n\leq x}|h_{(w,x^{1/2})}(P(Wn+b)))-F_{(w,x^{1/2})}(W)|+\frac{1}{x}\sum_{n\leq x}|h_{[x^{1/2},x]}(P(Wn+b)))-F_{[x^{1/2},x]}(W)|\\
 &+\frac{1}{x}\sum_{n\leq x}\sum_{\substack{p^k\mid \mid P(Wn+b)\\ w<p\leq x}}|h(p^k)|^2\\ 
 &\coloneqq S_1+S_2+S_3.
\end{align*}

 Note that by the irreducibility of $P$ we have $\rho(p^k)\ll 1$. Then by Cauchy--Schwarz,~\cite[Lemma 2.2]{klurman} and the prime number theorem,
\begin{align*}
|S_1|&\ll \Big(\frac{1}{x}\sum_{n\leq x}|h_{(w,x^{1/2})}(P(Wn+b)))-F_{(w,x^{1/2})}(W)|^2\Big)^{1/2}\\
&\ll \Big(\frac{1}{x}\sum_{n\leq x}\Big|h_{(w,x^{1/2})}(P(Wn+b)))-\sum_{\substack{p^k\leq x\\p>w}}h_{(w,x^{1/2})}(p^k)\left(\frac{\rho(p^k)}{p^k}-\frac{\rho(p^{k+1})}{p^{k+1}}\right)\Big|^2\Big)^{1/2}+w^{-1/2}\\
&\ll \Big(\sum_{\substack{p\in (w,x^{1/2})\\k\in \mathbb{N}}}\frac{|h(p^{k})|^2}{p^{k}}+w^{-1}+\frac{(\log \log x)^3}{\log x}\Big)^{1/2}\\
&\ll \Big(\sum_{p}\frac{|h(p)|^2}{p}\Big)^{1/2} +w^{-1/3}.
\end{align*}
Since $|h(p)|^2\leq 2-2\textnormal{Re}(g(p))$, we see that the contribution of $S_1$ is admissible.

We then bound $S_2$. By the triangle inequality and the prime number theorem, 
\begin{align*}
|S_2|\ll \frac{1}{x}\sum_{\substack{p\in (x^{1/2},x]\\k\in \mathbb{N}}}|h(p^{k})|\sum_{\substack{n\leq x\\p^{k}\mid P(Wn+b)}}1+\sum_{p\in [x^{1/2},x]}\frac{|h(p)|}{p}\ll \sum_{\substack{p\in (x^{1/2},x]\\k\in \mathbb{N}}}\frac{|h(p^{k})|\rho(p^k)}{p^k}+\frac{1}{\log x}.
\end{align*}
Using Cauchy--Schwarz, Mertens's theorem and the bound $\rho(p^k)\ll 1$, the main term here is
\begin{align*}
&\leq \Big(\sum_{\substack{p\in (x^{1/2},x]\\k\in \mathbb{N}}}\frac{|h(p^{k})|^2}{p^k}\Big)^{1/2}\Big(\sum_{\substack{p\in (x^{1/2},x]\\k\in \mathbb{N}}}\frac{\rho(p^{k})^2}{p^k}\Big)^{1/2}\\
&\ll \Big(\sum_{\substack{p\in (x^{1/2},x]\\k\in \mathbb{N}}}\frac{|h(p^{k})|^2}{p^k}\Big)^{1/2}\\
&\ll \Big(\sum_{p}\frac{|h(p)|^2}{p}\Big)^{1/2} +w^{-1/3},
\end{align*}
which is again an admissible contribution.

Lastly, by the prime number theorem we have
\begin{align*}
|S_3|\ll \sum_{\substack{p\in (w,x]\\k\in \mathbb{N}}}\frac{|h(p^k)|^2\rho(p^k)}{p^k}+\frac{1}{\log x}\ll \sum_{p}\frac{|h(p)|^2}{p}+w^{-1},   
\end{align*}
and this is again small enough.
Combining the bounds for $S_1,S_2,S_3$, the claim follows. 
\end{proof}

\begin{proof}[Proof of Theorem~\ref{theo_multiplicative-pret}]
We may assume that $g$ is unimodular, since if there exists a prime power $p^{\alpha}$ such that $|g(p^{\alpha})|=1-\delta'<1$ and $v_p(P(n))=\alpha$ for some $n$, then $|g(P(n))|\leq 1-\delta'$ for a positive lower density of $n$.

Let $\widetilde{\chi}$ be the completely multiplicative function given by $\widetilde{\chi}(p)=\chi(p)$ for $p\nmid q$ and $\widetilde{\chi}(p)=1$ otherwise, where $q=\textnormal{cond}(\chi)$. We claim that we may assume that $n\mapsto \widetilde{\chi}(P(n))$ is constant. Indeed, otherwise we may consider the polynomial $Q(n)=P(q^rn+\ell)$, where $\ell\geq 1$ is any integer such that $P(\ell)> 0$  and $r\in \mathbb{N}$ is such that $p^r\nmid P(\ell)$ for all $p\mid q$. Note that $\widetilde{\chi}(Q(n))=\widetilde{\chi}((q^r,Q(0)))$. The polynomial $Q$ has property $\mathcal{S}$ if $P$ has it, and since $n\mapsto g(p)^{v_p(P(n))}$ is non-constant for infinitely many $p$, the same is true for $Q$. Note also that if $|\frac{1}{x}\sum_{n\leq x}g(Q(n))|\leq 1-\delta'$, then $|\frac{1}{q^rx}\sum_{n\leq q^rx}g(P(n))|\leq 1-\delta'/(2q^r)$. 

We may also assume that $\chi\equiv 1$, since if we write $g=\widetilde{\chi}h$, then by the previous assumption for some constant $c\in \mathbb{D}$ we have $g(P(n))=c\cdot h(P(n))$ for all $n$, and now $h$ pretends to be $n^{it}$.

Let $g'$ be the completely multiplicative function given on the primes by $g'(p)=g(p)p^{-it}$. We may assume that $(P,g')$ is admissible, since if $t=0$ this reduces to our assumption that $(P,g)$ is admissible and if $t\neq 0$ and $(P,g')$ is not admissible, then for some constant $c\in \mathbb{D}$ we have $g(P(n))P(n)^{-it}=c$ for all $n$, in which case $|\frac{1}{x}\sum_{n\leq x}g(P(n))|=|\frac{1}{x}\sum_{n\leq x}cP(n)^{it}|\leq 1-\delta'$ for some constant $\delta'>0$.

Now, since $(P,g')$ is admissible, there exist a prime $p_0$ and two natural numbers $n_1\neq n_2$ such that $\kappa\coloneqq |g'(p_0^{v_{p_0}(P(n_1))})- g'(p_0^{v_{p_0}(P(n_2))})|>0$. Let $\ell\geq 1$ be any integer with $P(\ell)\neq 0$. Let $w$ be a large enough integer in terms of $P,\ell, n_1,n_2$, and let $W=\prod_{p\leq w}p^w$. Thus $W$ is a large constant. By the Chinese remainder theorem, we can find two distinct natural numbers $n_1',n_2'\leq W$ such that for all primes $p\leq w$ we have
\begin{align}\label{eq:turan5}
v_p(P(n_i'))=\begin{cases}v_{p_0}(P(n_i)),\quad p=p_0\\
v_p(P(\ell)),\quad\,\,\,\,\, p\neq p_0.
\end{cases}
\end{align}
In particular, if we define the completely multiplicative function $g_{\leq w}'$ by $g_{\leq w}'(p)=g'(p)$ for $p\leq w$ and $g_{\leq w}'(p)=1$ for $p>w$, then 
\begin{align}\label{e189}
|g_{\leq w}'(P(n_1'))- g_{\leq w}'(P(n_2'))|=\kappa.
\end{align}

 Define the sets
\begin{align*}
\mathcal{S}_{j}\coloneqq \{n\leq x\colon  \,\, P(Wn+j)\,\, \textnormal{is}\,\, x\,\textnormal{ smooth}\}. 
\end{align*}
Note that the number of $n\leq x$ for which $P(Wn+j)$ has a prime factor from $(x,W(x+1)]$ is 
\begin{align*}
\ll \sum_{x\leq p\leq W(x+1)}\left(\frac{x}{p}+1\right)\ll \frac{Wx}{\log x}.   
\end{align*}
Hence, by the assumption that $P$ has property $\mathcal{S}$, there exists an absolute constant $\eta_0>0$ such that
\begin{align}\label{e188}
|\mathcal{S}_{j}|\geq \eta_0 x    
\end{align}
for all large enough $x$.

Let $P=P_1^{m_1}\cdots P_{s}^{m_s}$ for some $m_j\in \mathbb{N}$ and pairwise coprime irreducible $P_j(x)\in \mathbb{Z}[x]$. Let $P_j$ have leading coefficient $c_j\geq 1$ and degree $d_j$. Let $g_{\leq x}$ be the completely multiplicative function defined on primes $p$ by $g_{\leq x}(p)=g(p)$ if $p\leq x$ and by $g_{\leq x}(p)=p^{it}$ if $p>x$. Recall that $\chi\equiv 1$ and observe that $\sum_{w\leq p\leq x}\rho(p)\frac{g_{\leq x}(p)p^{-it}-1}{p}=o_{w\to \infty}(1)$, where $\rho(r)$ denotes the number of solutions to the congruence $P(y)\equiv 0\pmod{r}$.
Then, by Lemma~\ref{le_concentration}, for any $l\in \{1,\ldots,m\}$ and $j\in \{1,2\}$ we have
\begin{align*}
 &\frac{1}{x}\sum_{\substack{n\leq x\\n\in \mathcal{S}_{n_j'}}}|g(P_l(Wn+n_j'))-(c_l(Wn)^{d_l})^{it}g_{\leq w}'(P_l(n_j'))|\\
 &=\frac{1}{x}\sum_{\substack{n\leq x\\n\in \mathcal{S}_{n_j'}}}|g_{\leq x}(P_l(Wn+n_j'))-(c_l(Wn)^{d_l})^{it}g_{\leq w}'(P_l(n_j'))\exp(F(W))|+o_{w\to \infty}(1)\\
 &=o_{w\to \infty}(1).
\end{align*}
From Markov's inequality and~\eqref{e188}, we conclude that there exist some sets $\widetilde{S}_{n_j'}\subset \mathcal{S}_{n_j'}$ with $|\widetilde{S}_{n_j'}|\geq (\eta_0-o_{w\to \infty}(1))x$ and such that for any $l\in \{1,\ldots,m\}$ and $j\in \{1,2\}$ we have the approximation
\begin{align*}
 g(P_l(Wn+n_j'))=(c(Wn)^d)^{it}g_{\leq w}'(P_l(n_j'))+o_{w\to \infty}(1) \,\, \textnormal{ for }\,\, n\in \widetilde{\mathcal{S}}_{n_j'}.   
\end{align*}
Using the inequality~\eqref{eq:turan0} and the complete multiplicativity of $g$, we see that for $\geq (\eta_0-o_{w\to \infty}(1))x$ natural numbers $n\leq x$ for either $j\in \{1,2\}$ we have
\begin{align}\label{e119}
g(P(Wn+n_j'))=\prod_{l=1}^s(c_l(Wn)^{d_{l}})^{itm_{l}}g_{\leq w}'(P(n_j'))+o_{w\to \infty}(1)\,\, \textnormal{ for }\,\, n\in \widetilde{\mathcal{S}}_{n_j'}.
\end{align}

If $t\neq 0$, then the claim of the theorem follows immediately from~\eqref{e119},~\eqref{e188},  and the fact that for any $\zeta\in \mathbb{D}$ the number of $n\leq x$ with $|(c(Wn)^d)^{it}-\zeta|\leq \varepsilon$ is $\ll_t \varepsilon x$ for $x$ large enough in terms of $\varepsilon$.  

Suppose then that $t=0$. Then the claim of the theorem follows from~\eqref{e119},~\eqref{e189} and the triangle inequality.
\end{proof}

We then outline the proof of Proposition~\ref{prop_smooth}, which complements Theorem~\ref{theo_multiplicative-pret}.

\begin{proof}[Proof sketch of Proposition~\ref{prop_smooth}]

The case $P(n)=(a_1n+h_1)\cdots (a_kn+h_k)$ is easy, since we may use the well-known fact that the set of numbers $n\leq x$ with a prime factor $>x^{1-\varepsilon}$ has density $o_{\varepsilon \to 0}(1)$, and combine this with the union bound.

In the case of irreducible quadratic polynomials, Harman~\cite{harman-2008} (improving on work of Dartyge~\cite{dartyge}) proved a much stronger result in a special case, namely that $P(n)=n^2+1$ is  $n^{4/5+\varepsilon}$-smooth for lower density $\geq \delta_0>0$ of numbers $n$.  Harman remarks that the same method works for $P(n)=n^2-D$, for any $D$ not a perfect square, with minor modifications (the function $\rho_D(p)$ defined as the number of roots of $x^2-D$ modulo $p$ has mean value $1/2$ by the Chebotarev density theorem, just as in the case of $\rho_1(p)$ in Harman's argument, so the half-dimensional sieve that Harman uses applies equally well in this case of arbitrary $D$). Since Harman's argument is sieve-theoretic, also a congruence condition $n\equiv b\pmod q$ may be incorporated, and it does not influence the constant in the lower bound. Now, for any quadratic polynomial $P(n)=an^2+bn+c$ we have $4aP(n)=(2an+b)^2-\Delta$, where $\Delta$ is the discriminant of $P$, so the smoothness of $P(n)$ for many $n$ is reduced to the same problem for the polynomial sequence $n^2-\Delta$ with $n\equiv b\pmod {2|a|}$. Therefore, by the preceding discussion, the proposition holds for arbitrary irreducible quadratics. 
\end{proof}

\section{Products of quadratic polynomials}

We shall now prove Theorem~\ref{theo_quadprod}. This proof is rather different from that of the other theorems, and it proceeds in the following steps. 

In Step 1, we interpret the condition $\lambda(P(n))=v$ as a recurrence relation, from which we are able to deduce that 
\begin{align*}
\eta(n)\coloneqq \lambda(n^2+1)    
\end{align*}
is an eventually periodic sequence.

In Step 2, we show that $\eta$ satisfies a functional equation in three variables arising from the multiplicativity of $\lambda$. This step involves applying quadratic reciprocity to prove the existence of solutions to a certain divisibility relation $F(x,y,z)\mid G(x,y,z)$, where $F$ and $G$ are certain polynomials and $x,y,z$ are restricted to any residue classes.

In Step 3, we solve this functional equation and find that the solutions are precisely functions of the form $\eta(n)=(-1)^{rn}\chi(n^2+1)$ with $r\in \{0,1\}$ and $\chi$ a real Dirichlet character (whose modulus only has prime factors $\equiv -1\pmod 4$); this is the lengthiest part of the argument and proceeds in several reductions.

In Step 4, we use the theory of the negative Pell equation (and in particular a theorem of Legendre on its solvability) to finally obtain a contradiction.

\subsection{Analyzing a recurrence relation}

Step 1 is implemented in the following lemma.
\begin{lemma}\label{le_periodic}
 Let $D,h_1,\ldots, h_k\in \mathbb{Z}$. Suppose that $\lambda(P(n))=v$ for all large enough $n$ and some $v\in \{-1,+1\}$, where $P(x)=\prod_{j=1}^k ((x+h_j)^2+D)$. Then the function $n\mapsto \lambda(n^2+D)$ is eventually periodic.   
\end{lemma}

\begin{remark}
This lemma holds for any $D\in \mathbb{Z}$, but later in the proof of Theorem~\ref{theo_quadprod}, particularly in step 4, the restriction to $D=1$ will be useful.
\end{remark}

\begin{proof}
Suppose for the sake of contradiction that
\begin{align*}
\lambda((n+h_1)^2+D)\cdots \lambda((n+h_k)^2+D)=v
\end{align*}
for all $n\geq N_0$. We may assume that $h_1<h_2<\ldots<h_k$. Now we have a recursion
\begin{align*}
\lambda((n+h_k)^2+D)=v\lambda((n+h_{k-1})^2+D)\cdots \lambda((n+h_1)^2+D)
\end{align*}
for $n\geq N_0$. We claim that any such recurrence sequence that takes values in $\{-1,+1\}$ is eventually periodic. Indeed, let $F\colon  \mathbb{N}\to \{-1,+1\}$ be any function that satisfies a recurrence of the form
\begin{align}\label{e80}
F(n+K)=G(F(n),F(n+1),\ldots, F(n+K-1))    
\end{align}
for all $n\geq N_0$ and some $K\geq 1$ and $G\colon  \{-1,+1\}^{K}\to \{-1+1\}$. Consider the vectors $u(n)\coloneqq (F(n),F(n+1),\ldots, F(n+K))$. By the pigeonhole principle, we may find $N_0<n_1<n_2$ such that $u(n_1)=u(n_2)$. But since $F$ satisfies the recurrence~\eqref{e80}, we deduce by induction on $n$ that $u(n+n_1)=u(n+n_2)$ for all $n$, so $F$ is periodic with period $n_2-n_1$ from some point onward.
\end{proof}

\subsection{Deriving a functional equation}

In this subsection, we will execute step 2 of the argument.

Applying Lemma~\ref{le_periodic} with $D=1$, we may assume from now on that for some constant $C$ and some periodic function $\eta\colon  \mathbb{Z}\to \{-1,+1\}$ of some period $q\geq 1$ we have 
\begin{align*}
\lambda(n^2+1)=\eta(n)     
\end{align*}
for all $n\geq C$. In what follows, we sometimes identify $\eta$ with a function on $\mathbb{Z}_q$.

A key observation in determining the function $\eta$ is that if $x,y,z$ are integers with  
\begin{align*}
xy+yz+zx-1\mid xyz-x-y-z,    
\end{align*}
then by the complete multiplicativity of $\lambda$ we have \begin{align}\label{e81}
\lambda(x^2+1)\lambda(y^2+1)\lambda(z^2+1)&=\lambda((x^2+1)(y^2+1)(z^2+1))\nonumber\\
&=\lambda((xyz-x-y-z)^2+(xy+yz+zx-1)^2)\nonumber\\
&=\lambda\left(\frac{(xyz-x-y-z)^2}{(xy+yz+zx-1)^2}+1\right).    
\end{align}
It is easy to check that the expression inside the argument in~\eqref{e81} is $>C$ whenever $|x|,|y|,|z|\geq 10C$, say. Hence, if additionally $|x|, |y|, |z|\geq 10C$, we must have
\begin{align}\label{e15}
\eta(|x|)\eta(|y|)\eta(|z|)=\eta\left(\left|\frac{xyz-x-y-z}{xy+yz+zx-1}\right|\right),\,\, \textnormal{when}\,\,\, xy+yz+zx-1\mid xyz-x-y-z.     
\end{align}

It will turn out to be convenient to initially restrict to even values of the argument only. Denoting
\begin{align}\label{e162}
\psi(x)=\eta(2|x|),
\end{align}
from~\eqref{e15} we see that
\begin{align}\label{e0}
\psi(x)\psi(y)\psi(z)=\psi\left(\frac{4xyz-x-y-z}{4(xy+yz+zx)-1}\right)\,\, \textnormal{when}\,\,\, 4(xy+yz+zx)-1\mid 4xyz-x-y-z.  
\end{align}

The divisibility condition for $x,y,z$ in~\eqref{e15} is complicating matters. However, we are able to dispose of this divisibility condition by proving that it has solutions $(x,y,z)$ in any given congruence class.

\begin{lemma}[Finding solutions to a divisibility relation]\label{le_divisible} Let $q\geq 1$. Let $a_1,a_2,a_3$ be any integers with $(4(a_1a_2+a_2a_3+a_3a_1)-1,q)=1$, and let $C\geq 1$ be any constant. Then there exist integers $x_1,x_2,x_3$ with $|x_i|\geq C$ such that $x_i\equiv a_i\pmod q$ for all $i\in \{1,2,3\}$ and
\begin{align*}
4(x_1x_2+x_2x_3+x_3x_1)-1\mid 4x_1x_2x_3-x_1-x_2-x_3.    
\end{align*}
\end{lemma}

\begin{proof}
Denoting $A\coloneqq 4(a_1a_2+a_2a_3+a_3a_1)$,  by assumption we have $(A-1,q)=1$.
Moreover, by permuting the variables we may assume that $a_1+a_2\neq 0$. Choose $k$ such that $|a_1+a_2|<q^k$. Let $d=(a_1+a_2,q^k)$. 

Let $r,p$ be any large primes satisfying
\begin{align}\label{e215}\begin{split}
 p&\equiv \frac{a_1+a_2}{d}\pmod{qd}\\
 r&\equiv 1-A\pmod{4qd}\\
 r&\equiv 1-b_p\pmod p,
 \end{split}
\end{align}
where $b_p\not \equiv 1\pmod p$ is a quadratic residue $\pmod p$.
There exist infinitely many such pairs $(r,p)$, since Dirichlet's theorem and $((a_1+a_2)/d,q)=1$ imply that there are infinitely many solutions to the first congruence, and once $p$ is fixed, the Chinese remainder theorem and Dirichlet's theorem imply that there are infinitely many choices for $r$.

Let $x_1$ be any integer satisfying
\begin{align}\label{e216}
x_1\equiv a_1\pmod{4qd},\quad 4x_1^2\equiv -1\pmod r;    
\end{align}
this pair of congruences has infinitely many solutions, since $r\equiv 1\pmod 4$ is a large prime. 

We further set
\begin{align*}
x_2=dp-x_1,\quad x_3=\frac{(1-r)/4-x_1(dp-x_1)}{dp},   
\end{align*}
so that
\begin{align}\label{e214}
4(x_1x_2+x_2x_3+x_3x_1)-1=-r,\quad x_1+x_2=dp. \end{align}

We now claim that $(x_1,x_2,x_3)$ is a tuple with all the desired properties, provided that $|x_1|>dpr$. 

Firstly, since $r$ and $p$ are large, we have
$|x_2|=|dp-x_1|\geq p>C$ and $|x_3|\geq r/2>C$.

Secondly, working modulo $q$, and recalling~\eqref{e215},~\eqref{e216}, we have
\begin{align*}
x_2\equiv dp-x_1\equiv (a_1+a_2)-a_1\equiv a_2\pmod q    
\end{align*}
and 
\begin{align*}
x_3&=\frac{(1-r)-4x_1(dp-x_1)}{4dp}\equiv \frac{4(a_1a_2+a_2a_3+a_3a_1)-4a_1a_2}{4d}p^{-1}\\
&\equiv \frac{a_3(a_1+a_2)}{d}\left(\frac{a_1+a_2}{d}\right)^{-1}\equiv a_3\pmod q.    
\end{align*}

Lastly, recalling~\eqref{e214}, we need to verify that
\begin{align}\label{e213a}
   \frac{1-r}{4}-x_1(dp-x_1)\equiv 0\pmod{dp}, 
\end{align}
and 
\begin{align}
    \label{e213b}
     (4x_1(dp-x_1)-1)\frac{(1-r)/4-x_1(dp-x_1)}{dp}-dp\equiv 0\pmod r.
\end{align}

For verifying~\eqref{e213a}, note that it is equivalent to 
\begin{align}\label{e217}
x_1^2\equiv \frac{1-r}{4}\pmod{dp}    
\end{align}
The congruence here is solvable $\pmod p$, since by~\eqref{e215} we have $1-r\equiv b_p\pmod p$ for some quadratic residue $b_p\pmod p$. The congruence is also solvable modulo $d$, since $a_1+a_2\equiv 0\pmod d$ implies that
\begin{align*}
 x_1^2-\frac{1-r}{4}\equiv a_1^2+(a_1a_2+a_2a_3+a_3a_1)\equiv a_1^2+a_1a_2\equiv 0\pmod d.   
\end{align*}

For verifying~\eqref{e213b}, note that since $r$ is a large prime, it is equivalent to
\begin{align*}
 (4x_1(dp-x_1)-1)((1-r)-4x_1(dp-x_1))-4(dp)^2&\equiv 0\pmod r    
\end{align*}
and since $1-r\equiv 1\pmod r$ this is further equivalent to
\begin{align}\label{e218}
 (4x_1(dp-x_1)-1)^2\equiv -4(dp)^2\pmod r.   
\end{align}
Now, as $4x_1^2\equiv -1\pmod r$ by~\eqref{e216}, we can easily simplify the left-hand side of~\eqref{e218} to obtain the right-hand side. Now we have inspected all the required properties for $(x_1,x_2,x_3)$, so the lemma follows.
\end{proof}

By Lemma~\ref{le_divisible}, given any $x,y,z\in \mathbb{Z}_q$ with $(4(xy+yz+zx)-1,q)=1$, we can find integers $x',y',z'$ of absolute value $\geq 10C$ such that $x\equiv x'\pmod q$, $y\equiv y'\pmod q$, $z\equiv z'\pmod q$ and
\begin{align*}
4(x'y'+y'z'+z'x')-1  \mid    4x'y'z'-x'-y'-z'. 
\end{align*}
Then,  recalling the definition of $\psi$ in~\eqref{e162} (and interpreting $\psi$ as a function on $\mathbb{Z}_q$), from~\eqref{e0} we see that
\begin{align}\label{e2}
\psi(x)\psi(y)\psi(z)=\psi\left(\frac{4xyz-x-y-z}{4(xy+yz+zx)-1}\right)\quad \textnormal{whenever}\quad (4(xy+yz+zx)-1,q)=1.     
\end{align}
Now we have removed the divisibility condition  from~\eqref{e0}, so we have derived a functional equation for $\psi$.

\subsection{Analyzing the functional equation}

We now proceed to step 3 of the argument. We claim that the solutions to the functional equation~\eqref{e2} are as follows.

\begin{proposition}[Solving the functional equation for $\psi$]\label{prop_chi}
If $\psi\colon  \mathbb{Z}_q\to \{-1,+1\}$ satisfies the functional equation~\eqref{e2} and $\psi(0)=1$, then there exists a real primitive Dirichlet character $\chi\pmod{q'}$ with $q'\mid q$ and $r\in \{0,1\}$ (with $r=0$ if $q$ is odd) such that $\psi(x)=(-1)^{rx}\chi(4x^2+1)$ for all $x\in \mathbb{Z}_q$. Moreover, $q'$ has all its prime divisors congruent to $-1\pmod 4$. Conversely, any such  function $\psi$ obeys~\eqref{e2}.
\end{proposition}

\begin{remark}
We restrict above to functions $\psi$ with $\psi(0)=1$, since clearly $\psi$ satisfies~\eqref{e2} if and only if $-\psi$ does.
\end{remark}

The sufficiency part of Proposition~\ref{prop_chi} is easy: Similarly as in~\eqref{e81}, we have 
\begin{align}\label{e200}
\chi(4x^2+1)\chi(4y^2+1)\chi(4z^2+1)=\chi\left(4\left(\frac{4xyz-x-y-z}{4(xy+yz+zx)-1}\right)^2+1\right),    
\end{align}
and working modulo $2$ we have
\begin{align}\label{e90}
(-1)^{x}(-1)^{y}(-1)^{z}=(-1)^{(4xyz-x-y-z)/(4(xy+yz+zx)-1)}.     
\end{align}
Furthermore, the condition that all the prime factors of $q'$ are $-1\pmod 4$ guarantees that $x\mapsto \chi(4x^2+1)$ is taking only values $\pm 1$. The much more challenging task will be to show the necessity part of Proposition~\ref{prop_chi}.

For later convenience, we make the following definition.

\begin{definition}[Primitive solutions]
We say that a solution $\psi\colon  \mathbb{Z}_q\to \{-1,+1\}$ to~\eqref{e2} is \emph{primitive} if the smallest period of $\psi$ is $q$.
\end{definition}

 If $\Psi\subset \{-1,+1\}^{\mathbb{Z}_q}$ is the class of all primitive solutions to~\eqref{e2}, then we claim that all the solutions are \emph{induced} by the functions in $\Psi$, where we say that $\psi\colon  \mathbb{Z}_q\to \{-1,+1\}$ is induced by $\psi^{*}\colon  \mathbb{Z}_{q'}\to \{-1+1\}$ if $q'\mid q$ and $\psi= \psi^{*}\circ \pi$, where $\pi$ is the canonical projection $\mathbb{Z}_{q}\to \mathbb{Z}_{q'}$. Indeed, the only nontrivial part of the above statement is that if $q'$ has fewer prime factors than $q$, then the condition $(4(xy+yz+zx)-1,q')=1$ in~\eqref{e2} is different from the corresponding condition with $q$ in place of $q'$, but by the Chinese remainder theorem if $x,y,z\in \mathbb{Z}_{q'}$, any congruence conditions on $x,y,z\pmod{\prod_{p\mid q, p\nmid q'}p^{N}}$ are harmless.

\subsubsection{Reduction to moduli having no prime factors $\equiv 1\pmod 4$}

In this subsection, we begin  the proof of Proposition~\ref{prop_chi} by showing that when solving~\eqref{e2} we may assume that $q$ has no prime factors $\equiv 1\pmod 4$.

\begin{lemma}[The modulus $q$ has no prime factors $\equiv 1\pmod 4$]\label{le_no1mod4}
If $q$ has a prime divisor $\equiv 1\pmod 4$, then there are no primitive solutions $\psi\colon  \mathbb{Z}_q\to \{-1,+1\}$ to~\eqref{e2}.
\end{lemma}

\begin{proof}
Let us write $q=q_1q_2$, where $q_1$ has only prime factors $\equiv 1\pmod 4$ and $q_2$ has only prime factors $\not \equiv 1\pmod 4$.  Our goal is to show that $x\mapsto \psi(q_2x+b)$ is constant for every $b$, which then contradicts the primitivity of $\psi$.

Since $q_1$ has only prime factors $\equiv 1\pmod 4$, we can choose an integer $x_0$ such that $4x_0^2+1\equiv 0\pmod{q_1}$.

For every integer $b$, fix an integer $c$ satisfying $(4(x_0b+bc+cx_0)-1,q_2)=1$ and $(c+x_0,q_1)=1$; since $(q_1,q_2)=1$, such an integer exists by the Chinese remainder theorem (note that if $b\equiv -x_0\pmod{p}$ for some prime $p\mid q_2$, then $4(x_0b+bc+cx_0)-1\equiv -4x_0^2-1\not \equiv 0\pmod p$ since $q_2$ has no prime factors $\equiv 1\pmod 4$). Then, take $x=x_0$, $y\equiv b\pmod{q_2}$, $z=c$ in~\eqref{e2} to see that
\begin{align}\label{e207}
\psi(x_0)\psi(y)\psi(c)=\psi\left(\frac{4x_0yc-x_0-y-c}{4(x_0y+yc+cx_0)-1}\right)\,\, \textnormal{whenever}\,\, (4(x_0y+yc+cx_0)-1,q_1)=1.
\end{align}
Let $Z$ denote the argument of $\psi$ on the right-hand side of~\eqref{e207}. By the property $4x_0^2\equiv -1\pmod{q_1}$ and simple algebra, we see that
\begin{align*}
Z\equiv x_0\pmod{q_1},\quad Z\equiv \frac{4x_0bc-x_0-b-c}{4(x_0b+bc+cx_0)-1}\pmod{q_2}.
\end{align*}
Hence, by the Chinese remainder theorem, $Z\pmod{q}$ depends only on $b$ (recalling that $c$ was fixed in terms of $b$). Returning to~\eqref{e207}, we conclude that for every $b$ there exists some $\gamma_{b}\in \{-1,+1\}$ such that
\begin{align*}
 \psi(y)=\gamma_{b}\,\,\textnormal{whenever}\,\,   y\equiv b\pmod{q_2},\,\, \textnormal{and}\,\, (4(x_0y+yc+cx_0)-1,q_1)=1. 
\end{align*}
Since $(q_1,q_2)=1$ and $(c+x_0,q_1)=1$, by the Chinese remainder theorem each residue class $b\pmod{q_2}$ contains an element $y$ satisfying $(4(x_0y+yc+cx_0)-1,q_1)=1$. Hence, $\psi(y)=\gamma_b$ for all $y\equiv b\pmod{q_2}$. This means that $x\mapsto \psi(q_2x+b)$ is constant for every $b$, which gives the desired contradiction.
\end{proof}

\subsubsection{Reduction to squarefree periods}

Having reduced the analysis of the functional equation~\eqref{e2} to those $q$ that have no prime factors $\equiv 1\pmod 4$, we will next show that the odd part of $q$ is squarefree.

\begin{lemma}[The modulus $q$ is essentially squarefree]\label{le_squarefree}
If $\psi\colon  \mathbb{Z}_q\to \{-1,+1\}$ satisfies~\eqref{e2} and  $\psi$ is primitive, then $q/2^{v_2(q)}$ is squarefree. 
\end{lemma}

\begin{proof}
Suppose $q/2^{v_2(q)}$ is not squarefree. Let $p$ be an odd prime such that $p^2\mid q$, and let $q'=q/p$. Recall also that by Lemma~\ref{le_no1mod4} $q$ has no prime factors $\equiv 1\pmod 4$.

We want to show that $x\mapsto \psi(2q'x+a)$ is constant for every $a$, which will then contradict the primitivity of $\psi$ as $2q'<q$.

We substitute $(q'x,q',0)$ in place of $(x,y,z)$ in~\eqref{e2} and crucially use the fact that $(q')^2\equiv 0\pmod q$ to simplify the resulting expression, obtaining \begin{align}\label{e194}
 \psi(q'x)\psi(q')\psi(0)=\psi(q'(x+1)).   
\end{align}
Iterating this once, we see that $x\mapsto \psi(q'x)$ is $2$-periodic, so in particular 
\begin{align}\label{e196}
\psi(2q'x)\equiv c    
\end{align}
for all $x$ and some $c\in \{-1,+1\}$.

Now put $(2q'x+a,-a,0)$ in place of $(x,y,z)$ in~\eqref{e2} to obtain
\begin{align}\label{e210}
\psi(2q'x+a)\psi(-a)\psi(0)=\psi\left(\frac{-2q'x}{-4a(2q'x+a)-1}\right)    
\end{align}
Let $Z$ be the argument on the right-hand side. Then $Z$ is always well-defined, since $-4a(2q'x+a)-1\equiv -4a^2-1\not \equiv 0\pmod{q'}$ (as $q'$ has no prime factors $\equiv 1\pmod 4$). Moreover, $Z\equiv 0\pmod{2q'}$, so by~\eqref{e196} and~\eqref{e210} we have
\begin{align*}
\psi(2q'x+a)=\psi(0)\psi(-a)c  
\end{align*}
for all $x$ and $a$. But now $\psi$ is $2q'$-periodic, which is the desired contradiction. 
\end{proof}

We have now shown that the odd part of $q$ is squarefree; the same argument would not have worked for the $2$-part of $q$. However,  the following lemma shows that the $2$-part of $q$ can in fact be assumed to be trivial, after possibly multiplying $\psi(x)$ by $(-1)^x$.

\begin{lemma}[The modulus $q$ can be assumed odd]\label{le_odd} If $q$ is even and $\psi\colon  \mathbb{Z}_q\to \{-1,+1\}$ satisfies~\eqref{e2} and $\psi(0)=1$, then there exists $r\in \{0,1\}$ such that the function $\psi(x)(-1)^{rx}$ satisfies~\eqref{e2} and has period $q/2$.
\end{lemma}

\begin{proof}
From~\eqref{e90} we already see that $\psi(x)(-1)^{rx}$ satisfies~\eqref{e2}. Thus, what remains to be shown is that it is $q/2$-periodic for suitable $r\in \{0,1\}$.

Suppose first that $2\mid \mid q$. Take $y=q/2$, $z=0$ in~\eqref{e2}, and recall that $\psi(0)=1$ and $2(q/2)\equiv 0\pmod q$. The outcome is that
\begin{align*}
\psi(x)\psi(q/2)=\psi\left(x+q/2\right)   
\end{align*}
for any $x$. If $\psi(q/2)=1$, we see from this that $\psi$ is $q/2$-periodic. If in turn $\psi(q/2)=-1$, then (as $q/2$ is odd) we see that $\psi(x)(-1)^{x}$ is $q/2$-periodic. 

If instead $4\mid q$, we take $(x,y,z)=(x,q/4,q/4)$ in~\eqref{e2} to obtain
\begin{align*}
\psi(x)\psi(q/4)^2=\psi(x+q/4+q/4),    
\end{align*}
that is, $\psi(x)=\psi(x+q/2)$. Hence, in either case the claim follows.
\end{proof}

\subsubsection{Handling the remaining case}

We now turn to the final task of step 3, which is solving~\eqref{e2} for $\psi$ when $q$ is squarefree, odd and has no prime divisors $\equiv 1\pmod 4$ (with the oddness following from Lemma~\ref{le_odd}, after possibly replacing $\psi(x)$ with the new function $\psi(x)(-1)^x$). We split this task into two parts, the first of which involves making suitable substitutions to the functional equation~\eqref{e2}, whereas the second one involves counting points on curves in $\mathbb{Z}_p$.

\begin{lemma}[Values of $\psi$ at quadratic residues]\label{le_firsthalf} Let $q$ be squarefree, and suppose that all the prime factors of $q$ are $\equiv -1\pmod 4$. If $\psi\colon  \mathbb{Z}_q\to \{-1,+1\}$ satisfies~\eqref{e2} and $\psi(0)=1$, then $\psi(x)=1$ whenever $4x^2+1$ is a quadratic residue $\pmod q$.
\end{lemma}

\begin{proof}
Setting $(x,y,z)=(x,-x,0)$ in~\eqref{e2} and observing that $(4x^2+1,q)=1$ for every $x$ by our assumption on $q$, and also using $\psi(0)=1$, we deduce that
\begin{align}\label{e95}
\psi(x)=\psi(-x) \,\, \textnormal{for all}\,\, x.   
\end{align}

Keeping~\eqref{e95} in mind and substituting both $(x,y,0)$ and $(1/(4x),1/(4y),0)$ into~\eqref{e2}, for any $x,y$ coprime to $q$, and comparing the results, we see that 
\begin{align}\label{e96a}
\psi(x)\psi(y)=\psi\left(-\frac{x+y}{4xy-1}\right)=\psi\left(-\frac{1/(4x)+1/(4y)}{4\cdot (1/(4x))(1/(4y))-1}\right)=\psi\left(\frac{1}{4x}\right)\psi\left(\frac{1}{4y}\right)     
\end{align}
whenever $(x,q)=(y,q)=(4xy-1,q)=1$. In particular, taking $y=\pm 1$ and using~\eqref{e95}, we see that $\psi(x)\psi(1/(4x))=\psi(1)\psi(1/4)$ whenever  $(x,q)=1$ and $\min\{(4x-1,q),(4x+1,q)\}=1$. Since $q$ is odd, by the Chinese remainder theorem we can choose for every $y$ some $x$ such that $(x,q)=\min\{(4x-1,q),(4x+1,q)\}=(4xy-1,q)=1$. Then returning to~\eqref{e96a} we reach
\begin{align}\label{e96}
 \psi(y)\psi\left(\frac{1}{4y}\right)=\psi(1)\psi\left(\frac{1}{4}\right)\,\, \textnormal{whenever}\,\, (y,q)=1.   
\end{align}
Next, taking $(x,y,z)=(x,-1/(4x),0)$ in~\eqref{e2} with $x$ and $q$ coprime, and recalling~\eqref{e95},~\eqref{e96}, we find
\begin{align*}
\psi(1)\psi\left(\frac{1}{4}\right)\psi(0)=\psi(x)\psi\left(-\frac{1}{4x}\right)\psi(0)=\psi\left(\frac{4x^2-1}{8x}\right).     
\end{align*}
Now we claim that for any $n\in \mathbb{Z}_q$ for which $4n^2+1$ is a quadratic residue we can write $n\equiv (4x^2-1)/(8x)$ for some $(x,q)=1$. In other words, we claim that there is a  solution to the congruence $4x^2-8nx-1\equiv 0\pmod q$ (with the coprimality condition dropped, since it is automatically satisfied). Looking at the discriminant, this is solvable if and only if $(8n)^2+16$ is a square $\pmod q$, which is equivalent to our assumption on $n$. Thus, there exists a constant $\gamma$ such that $\psi(x)=\gamma$ whenever $4x^2+1$ is a square $\pmod q$. Since $\psi(0)=1$, we have $\gamma=1$, so the claim is proved.
\end{proof}

We are left with analyzing $\psi(x)$ when $4x^2+1$ is a quadratic \emph{nonresidue} $\pmod q$. For this we first need the following simple counting lemma. 

\begin{lemma}[Number of points on a hyperbola]\label{le_pointcount} Let $p\equiv -1\pmod 4$ be a prime. Then there are precisely $(p-1)/2$ values of $x\in \mathbb{Z}_p$ such that $4x^2+1\equiv y^2\pmod p$ for some $y\in \mathbb{Z}_p$.
\end{lemma}

\begin{proof}
Since $x\mapsto 2x\pmod p$ is a a bijection for any $p\neq 2$, it suffices to consider the congruence $x^2+1\equiv y^2\pmod p$. It is enough to show that
\begin{align}\label{e99}
 |\{(x,y)\in \mathbb{Z}_p\colon  \,\, x^2+1\equiv y^2\pmod p\}|=p-1, 
\end{align}
since whenever $x$ is such that $x^2+1\equiv y^2\pmod p$ is solvable, there are precisely two solutions $y\pmod p$, as $-1$ is not a square in $\mathbb{Z}_p$. Since $y^2-x^2=(y-x)(y+x)$ and $(x,y)\mapsto (y-x,y+x)$ is a bijection on $\mathbb{Z}_p^2$, we see that the cardinality~\eqref{e99} equals to
\begin{align*}
 |\{(z,w)\in \mathbb{Z}_p\colon  \,\, zw\equiv 1\pmod p\}|, 
\end{align*}
which clearly equals to $p-1$. 
\end{proof}

We are finally in a position to prove Proposition~\ref{prop_chi}.

\begin{proof}[Proof of Proposition~\ref{prop_chi}.]
In view of Lemma~\ref{le_firsthalf}, we know that 
\begin{align}\label{e197}
\psi(x)=1\,\, \textnormal{whenever}\,\, 4x^2+1 \hspace{-0.2cm}\pmod q\,\, \textnormal{is a quadratic residue},    
\end{align}
so what remains to be shown is that $\psi(x)=-1$ whenever $4x^2+1\pmod q$ is a quadratic nonresidue.

Note that by the same deduction as in the proof of Lemma~\ref{le_firsthalf}, 
\begin{align}
4x^2+1\pmod q\,\,\textnormal{is a square}\,\, \iff x\in \mathcal{U}_q\coloneqq \{(4t^2-1)/(8t)\colon  \,\, t\in \mathbb{Z}_q^{\times}\}    
\end{align} 
Now, choose $x_0$ such that $\psi(x_0)=-1$; if no such $x_0$ exists, there is nothing to prove. Taking $x=x_0$, $y\equiv (4t^2-1)/(8t)\pmod q$, $z=0$ in~\eqref{e2}, we find
\begin{align}\label{e198}
-1=\psi\left(\frac{-4t^2+1-8tx_0}{4(4t^2-1)x_0-8t}\right),   
\end{align}
whenever $4(4t^2-1)x_0-8t$ is coprime to $q$. Let $R(t)$ be the rational function on the right-hand side of~\eqref{e198}. Let $\mathcal{R}_p$ be its image $\pmod p$. By~\eqref{e197} and~\eqref{e198}, $4R(t)^2+1$ is a quadratic nonresidue $\pmod q$ whenever it is well-defined, so for all $p\mid q$ we have
\begin{align}\label{e199}
\mathcal{R}_p\subset \mathcal{V}_p\coloneqq \mathbb{Z}_p\setminus \mathcal{U}_p .   
\end{align} 
By Lemma~\ref{le_pointcount}, we deduce $|\mathcal{R}_p|\leq |\mathcal{V}_p|=(p+1)/2$.

On the other hand, we have $|\mathcal{R}_p|\geq (p-1)/2$ for any $p\mid q$, since the function $R(t)\pmod p$ is undefined in at most two points, say $t_1,t_2$, and the restriction $R\colon  \mathbb{Z}_p\setminus \{t_1,t_2\}\to \mathbb{Z}_p$ takes any value at most twice (one easily checks\footnote{Indeed, if for some $\gamma$ we had $-4t^2+1-8tx_0\equiv \gamma (4(4t^2-1)x_0)-8t\pmod p$ for all $t$, then comparing the coefficients we would obtain $4x_0^2\equiv -1\pmod p$, which is impossible as $p\equiv -1\pmod 4$.} that $R(t)\pmod p$ is not constant, so $R(t)\equiv t_0\pmod p$ always has at most two roots).  Comparing the upper and lower bounds for $|\mathcal{R}_p|$ and recalling~\eqref{e199}, we conclude that for some numbers $a_p$ we have
\begin{align*}
\mathcal{R}_p\supset \mathcal{V}_p\setminus \{a_p\pmod p\}.   
\end{align*}

By the Chinese remainder theorem, $\mathcal{R}_q$ (respectively, $\mathcal{V}_q$) is the Cartesian product of $\mathcal{R}_p$ (respectively, $\mathcal{V}_p$) over $p\mid q$, so we must have $\psi(x)=-1$ for all $x\in \mathcal{V}_q$ that satisfy $x\not \equiv a_p\pmod p$ for all $p\mid q$. Combining this with Lemma~\ref{le_firsthalf}, and denoting by $\chi$ the Legendre symbol $\pmod q$, we have 
\begin{align}\label{e201}
\psi(x)=\chi(4x^2+1)\,\, \textnormal{whenever}\,\, x\not \equiv a_p\pmod p\,\,\forall\,\ p\mid q.    
\end{align}

Let $\psi'(x)\coloneqq \psi(x)\chi(4x^2+1)$. Observe that by~\eqref{e200} $\psi'$ also obeys the functional equation~\eqref{e2}. Let $x_0'$ be such  that $\psi'(x_0')=-1$ (if no such $x_0'$ exists, we are done). Let $y\not \equiv a_p\pmod p$ for all $p\mid q$. Then substituting $(x_0',y,0)$ in place of $(x,y,z)$ in~\eqref{e2} (with $\psi$ replaced with $\psi'$ in~\eqref{e2}), from~\eqref{e201} we obtain
\begin{align*}
-1=\psi'\left(\frac{-x_0'-y}{4x_0'y-1}\right)
\end{align*}
To get a contradiction from this, it suffices to show that there is an integer $y$ for which  $y,(-x_0'-y)/(4x_0'y-1)\not \equiv a_p\pmod p$ for all $p\mid q$, as then $\psi'\left(\frac{-x_0'-y}{4x_0'y-1}\right)=1$. Note that $(-x_0'-y)/(4x_0'y-1)\not \equiv a_p\pmod p$ unless $y\equiv b_p\pmod p$ for some $b_p$, so by the Chinese remainder theorem we can choose $y$ that satisfies $y\not \equiv a_p,b_p\pmod p$ for all $p\mid q$. The proof is now complete.
\end{proof}

\subsection{Negative Pell equations}

We then implement step 4, the last step in the proof of Theorem~\ref{theo_quadprod}. It is here that the assumption that we are working with shifted products of the polynomial $x^2+1$ (as opposed to $x^2+D$) is quite useful, since otherwise instead of the negative Pell equation $x^2-py^2=-1$ we would encounter generalized Pell equations $x^2-py^2=D$, and it seems harder in this general case to show a property which we will need, namely that there exists a  congruence class $a\pmod c$ such that the equation is always solvable when $p\equiv a\pmod c$.

\begin{lemma} Let $r,s\in \{0,1\}$ and $q\geq 1$ an integer all of whose prime factors are $\equiv -1\pmod 4$. There is no real Dirichlet character $\chi \pmod q$ such that $\lambda(4n^2+1)=(-1)^{rn+s}\chi(4n^2+1)$ for all large enough $n$.
\end{lemma}

\begin{proof}
Suppose for the sake of contradiction that such a character $\chi\pmod q$ exists.

If $q=1$, then $\lambda(16n^2+1)$ is constant for all large enough $n$. Note that the generalized Pell equations
\begin{align}\label{e20}
16n^2+1=17y^2,\quad 16n^2+1=5\cdot 13y^2 
\end{align}
both have a solution (one can take $(n,y)$ to be $(1,1)$ and $(2,1)$, respectively). But by the theory of Pell equations  (see~\cite[Theorem 1]{borwein-choi-ganguli}) this implies that  both equations in~\eqref{e20} have infinitely many positive integer solutions, so $\lambda(16n^2+1)$ is not eventually constant. We may henceforth assume that $q\geq 2$.

By a theorem of Legendre (see~\cite[Proposition 2.1]{legendre}), the negative Pell equation 
\begin{align}\label{e19}
x^2-py^2=-1    
\end{align} 
has a positive integer solution $(x,y)$ for any given prime $p\equiv 1\pmod 4$. By the theory of Pell equations (\cite[Theorem 1]{borwein-choi-ganguli}), this implies that $x^2-py^2=-1$ has infinitely many positive integer solutions for such $p$. Working modulo $4$, we see that any solution to~\eqref{e19} must satisfy $x\equiv 0\pmod 2$, and thus there are infinitely many positive solutions $(n,y)$ to
\begin{align}\label{e19b}
4n^2+1=py^2.    
\end{align}
Moreover, any such solution satisfies $(y,q)=1$, since $-1$ is a quadratic nonresidue modulo any prime divisor of $q$. By the pigeonhole principle, we may pick $v\in \{0,1\}$ such that $n\equiv v\pmod 2$ for infinitely many solutions $(n,y)$ to~\eqref{e19b}.

But now~\eqref{e19b} implies that
\begin{align*}
(-1)^{rv+s}=\lambda\chi(4n^2+1)=\lambda\chi(py^2)=-\chi(p)    
\end{align*}
for any prime $p\equiv 1\pmod 4$. By Dirichlet's theorem, for any integer $a$ with  $a\equiv 1\pmod 4$ and $(a,q)=1$ there are infinitely many primes $p\equiv a\pmod q$, $p\equiv 1\pmod 4$. Hence $\chi(a)=(-1)^{rv+s+1}$ for all $(a,q)=1$, $a\equiv 1\pmod 4$. In particular, taking $a=1$ we see that $rv+s+1\equiv 0\pmod 2$, so $\chi(a)=1$ whenever $a\equiv 1\pmod 4$, $(a,q)=1$. Thus, $\chi(4x^2+1)=1$ for every $x$, so we conclude that  $(-1)^{2rn+s}\chi(4(2n)^2+1)=\lambda(16n^2+1)$ is eventually constant, which we already ruled out. This proves the lemma.
\end{proof}

\section{Reducible cubics}

In this section, we will prove Theorem~\ref{theo_cubic}. The proof strategy is to use some algebraic identities, which result from the multiplicativity of the norm map in $\mathbb{Q}(\sqrt{-d})$, to reduce the claim to Corollary~\ref{cor_linear}.

\begin{proof}[Proof of Theorem~\ref{theo_cubic}.]  Let $P(x)=x(x^2-Bx+C)$, and define the discriminant
\begin{align}\label{e32}
\Delta\coloneqq B^2-4C.    
\end{align}
We may assume that $\Delta$ is not a square, as otherwise $P(x)$ factors into linear factors over $\mathbb{Q}$, in which case we may directly apply Corollary~\ref{cor_linear}.

By shifting the $x$ variable, we may factorize
\begin{align*}
&8P\left(x+\frac{B}{2}\right)=(2x+B)((2x)^2-\Delta),\quad \textnormal{if}\quad B\equiv 0\pmod 2\\
&\textnormal{or}\quad 8P\left(x+ \frac{B+1}{2}\right)=(2x+1+B)((2x+1)^2-\Delta),\quad \textnormal{if}\quad B\equiv 1\pmod 2    
\end{align*}
 Denoting $Q(x)=(x+B)(x^2-\Delta)$, in either case we conclude that there exists a constant $L\geq 1$ such that for all $X\geq 1$ we have
\begin{align}\label{e211}
|\{n\leq X\colon  \,\,\lambda(P(n))=-v|\leq |\{n\leq 2X+L,n\equiv B\pmod 2\colon  \,\, \lambda(Q(n))=v\}|.
\end{align}

Let 
\begin{align*}
S_v&\coloneqq \{n\in \mathbb{N}\colon  \,\, \lambda(P(n))=v\},\\
S_v'&\coloneqq \{n\in \mathbb{N},n\equiv B\pmod 2\colon   \,\, \lambda(Q(n))=v\}.
\end{align*}
We assume for the sake of contradiction that 
\begin{align*}
|S_{-v}\cap [1,X]|\geq X-o(\sqrt{X})\quad \textnormal{for}\quad X\in \mathcal{X},    
\end{align*}
 where $\mathcal{X}\subset \mathbb{N}$ is some infinite set. Then by~\eqref{e211} we have
 \begin{align*}
 |S'_{v}\cap [1,2X]|\geq X-o(\sqrt{X})\quad \textnormal{for}\quad X\in \mathcal{X}.        
 \end{align*}
 The key observation is that, for any even integer $k\neq 0$ and for $n>|k|$ satisfying\\
(i) $k\mid n(n+k)-\Delta$;\\
(ii) $n,n+k,(n(n+k)-\Delta)/|k|\in S'_{v}$,\\
(iii) $n\equiv B\pmod 2$, $(n(n+k)-\Delta)/|k|\equiv B\pmod 2$,\\
we have 
\begin{align}\label{e212}\begin{split}
\lambda(n+B)\lambda(n+B+k)&=\lambda(n^2-\Delta)\lambda((n+k)^2-\Delta) \\
&=\lambda((n^2-\Delta)((n+k)^2-\Delta))\\
&=\lambda((n(n+k)-\Delta)^2-\Delta k^2)\\
&=\lambda\left(\left(\frac{n(n+k)-\Delta}{|k|}\right)^2-\Delta\right)\\
&=v \lambda\left(\frac{n(n+k)-\Delta}{|k|}+B\right)\\
&=v\lambda(|k|)\lambda(n(n+k)+B|k|-\Delta),
\end{split}
\end{align}
using properties (ii)--(iii) on the first and fifth lines.
We wish to choose $k$ so that the discriminant of $x(x+k)+B|k|-\Delta$ is a square, since then that polynomial factorizes over $\mathbb{Z}[x]$. We thus wish to find an integer solution $(k,y)$ to
\begin{align*}
k^2-4B|k|+4\Delta=y^2,    
\end{align*}
or equivalently to
\begin{align}\label{e36}
(|k|-2B)^2-y^2=-4\Delta+(2B)^2=16C.   
\end{align}
This equation is satisfied if
\begin{align}\label{e33}
|k|=2B+2C+2,y=2C-2\quad \textnormal{or}\quad |k|=2B-2C-2, y=-2C+2.    
\end{align}
Since $B\geq 0$, the first possibility in~\eqref{e33} is acceptable whenever $C\geq 0$ (and we can take $k=2B+2C+2>0$). If in turn $C<0$, the second possibility in~\eqref{e33} is acceptable (and we can take $k=2B-2C-2>0$), unless $C=-1,B=0$. But in this special case $\Delta=4$ is a perfect square, which was excluded.   

Now that we have found a suitable value of $k$, we will show that conditions (i) and (iii) hold for all large $n$ belonging to some residue class. Thus we need $n^2-\Delta\equiv 0\pmod{|k|}$, $n\equiv B\pmod 2$ and $n(n+k)-\Delta\equiv B|k|\pmod{2|k|}$. Note that all three conditions hold if 
\begin{align}\label{e161}
 n^2-\Delta\equiv 0\pmod{2|k|}.    
\end{align}
Hence, it suffices to find a solution to~\eqref{e161}.

If $y$ is as in~\eqref{e36}, keeping in mind that $k$ is even, we have 
\begin{align*}
y^2-4\Delta\equiv \begin{cases}0\pmod{4|k|},\,\,k\equiv 0\pmod 4\\0\pmod{2|k|},\,\, k\equiv 2\pmod 4.\end{cases}
\end{align*}
Now, if $8\mid k$, then we can divide this congruence by $4$ to see that $x^2\equiv \Delta \pmod{|k|}$ is solvable, and moreover, by Hensel's lemma the solvability of $x^2\equiv \Delta\pmod{8}$ implies the solvability of $x^2\equiv \Delta\pmod{2^{\alpha}}$ for all $\alpha\geq 0$, so $x^2\equiv \Delta \pmod{2|k|}$ is also solvable. If in turn $2\mid \mid k$, we see that $x^2\equiv \Delta \pmod{|k|/2}$ is solvable. But also the congruence $x^2\equiv \Delta \pmod 4$ is solvable, as necessarily $\Delta\equiv B^2\in \{0,1\}\pmod 4$. Hence, $x^2\equiv \Delta \pmod{2|k|}$ is solvable also in this case. Lastly, suppose that $4\mid \mid k$. Since $|k|=2B\pm(2C+2)$, we have $B\not \equiv C\pmod 2$. But then $\Delta=B^2-4C\in \{0,1,4\}\pmod 8$. Therefore, $x^2\equiv \Delta\pmod 8$ is solvable, so $x^2\equiv \Delta\pmod{2|k|}$ is solvable as well.  

Now consider the set 
\begin{align*}
T_{v}\coloneqq \{n\equiv n_0\pmod{2|k|}\colon  \,\, n,n+k, (n(n+k)-\Delta)/|k| \in S'_{v}\},
\end{align*}
where $n_0$ is a solution to $x^2\equiv \Delta\pmod{2|k|}$. Since $S'_{v}\cap [1,2X]$ contains $\geq X-o(\sqrt{X})$ elements for $X\in \mathcal{X}$, the set $T_{v}$ contains $(1-o(1))\sqrt{X}/(2|k|)$ numbers $n\leq  \sqrt{X}$, $n\equiv n_0 \pmod{2|k|}$ for $X\in \mathcal{X}$. Recall~\eqref{e212} and~\eqref{e36}. For $n\in T_v$, the discriminant of $x(x+k)+B|k|-\Delta$ is a square, so there exist integers $t_1,t_2$ (independent of $n$) such that
\begin{align*}
\lambda(n+B)\lambda(n+B+k)=v\lambda(|k|)\lambda(n+t_1)\lambda(n+t_2),\quad n\in T_v.    
\end{align*}
 Thus, we have
\begin{align*}
&\mathbb{E}_{m\leq \sqrt{X}/(2|k|)}\lambda(2|k|m+n_0+t_1)\lambda(2|k|m+n_0+t_2)\lambda(2|k|m+n_0+B)\lambda(2|k|m+n_0+B+k)\\
&=v\lambda(|k|)+o(1)    
\end{align*}
for $x\in \mathcal{X}$. However, this contradicts Corollary~\ref{cor_linear}, once we show that the polynomial
\begin{align}\label{eq3}
(2|k|x+n_0+t_1)(2|k|x+n_0+t_2)(2|k|x+n_0+B)(2|k|x+n_0+B+k)    
\end{align} 
is not a square of another polynomial. For this it suffices to show that 
\begin{align*}
|\{t_1,t_2,B,B+k\}|>2.    
\end{align*} 
Since $B\neq B+k$, and since $t_1\neq t_2$ by the assumption that $\Delta$ is not a perfect square, the claim above can be false only if $B=t_1$, $B+k=t_2$ or $B=t_2$, $B+k=t_1$. Both cases are similar; let us consider the first case. In this case, we also have $t_2-t_1=k$. But, by Vieta's formulas, we have $t_1+t_2=-k$, so  $(t_1,t_2)=(-k,0)$. But now $B=-k<0$, which contradicts the fact that $k$ was chosen to be positive and $B\geq 0$. \end{proof}

\section{Almost all polynomials}

We shall now prove Theorem~\ref{thm_almostall}. The proof of part (i) is based on the Gowers uniformity of non-pretentious multiplicative functions, established by Frantzikinakis and Host~\cite{fh-jams}, whereas the proof of part (ii) applies Tao's two-point logarithmic Elliott conjecture~\cite{tao-chowla}.

\begin{proof}[Proof of Theorem~\ref{thm_almostall}] (i) 
Let $N$ be a parameter tending to infinity. By the union bound, it suffices to prove for every fixed $v$ with $v^q=1$ that $g(P(n))=v$ for some $n$ for almost all degree $d$ polynomials with positive leading coefficient and of height $\leq N$.

Given a tuple $\bm{a}=(a_d,\ldots, a_1,a_0)\in [-N,N]^{d+1}$, let $P_{\bm{a}}(x)\coloneqq a_dx^d+\cdots +a_1x+a_0$ be the corresponding polynomial. Let $\mathcal{E}_N$ be the set of tuples in $\bm{a}\in [-N,N]^{d+1}$ for which $g(P_{\bm{a}}(n))\neq v$ for all $n\geq 1$ (recall that we extend all multiplicative functions as even functions to the negative integers). Suppose for the sake of contradiction that 
\begin{align*}
|\mathcal{E}_N|\geq \delta N^{d+1}    
\end{align*}
for some absolute constant $\delta>0$ and for $N\in \mathcal{N}$, where $\mathcal{N}\subset \mathbb{N}$ is some infinite set. Let $H=H(N)$ be a function tending to infinity extremely slowly.

Then for any $\bm{a}\in \mathcal{E}_N$ we have
\begin{align*}
\prod_{1\leq h\leq H}1_{g(P_{\bm{a}}(h))\neq v}=1.  
\end{align*}
Summing this over $a_i$, we see that
\begin{align}\label{eq10}
\sum_{a_0,a_1,\ldots, a_{d}\in [-N,N]}\prod_{1\leq h\leq H}1_{g(P_{\bm{a}}(h)))\neq v}\geq \delta N^{d+1}.     
\end{align}

Since $g$ takes values in the $q$th roots of unity, we have
\begin{align}\label{eq13}
1_{g(n)\neq v}=1-\frac{1}{q}\sum_{j=0}^{q-1}(\overline{v}g(n))^j,    
\end{align}
and substituting this to~\eqref{eq10} yields
\begin{align*}
\sum_{a_0,a_1,\ldots, a_{d}\in [-N,N]}\,\,\sum_{[1,H]=S_0\sqcup S_1\sqcup\cdots \sqcup S_{q-1}}\prod_{h\in S_0}\left(1-\frac{1}{q}\right)\prod_{1\leq j\leq q-1}\prod_{h\in S_j}-\frac{\overline{v}^j}{q}g(P_{\bm{a}}(h))^{j}\geq \delta N^{d+1}.  
\end{align*}
By the pigeonhole principle, we can find a fixed tuple $(a_0',\ldots, a_{d-2}')$ and a fixed tuple $(S_0',\ldots, S_{q-1}')$ such that
\begin{align}\label{e11b}
&\left|\sum_{x,y\in [-N,N]}\prod_{1\leq j\leq q-1}\prod_{h\in S_j'}-\frac{\overline{v}^j}{q}g\left(h^d x+h^{d-1}y+\sum_{j=0}^{d-2} a_j'h^j\right)^{j}\right|\geq q^{-H}\delta^2 N^{2}     
\end{align}
We claim that the tuple
\begin{align*}
\Psi\coloneqq \left(h^dy+h^{d-1}x\sum_{j=0}^{d-2}a_j'h^j\right)_{h\in [H]}    
\end{align*}
of $H$ linear forms in two variables has Cauchy--Schwarz complexity $\leq H$ (in the sense of~\cite[Definition 1.5]{gt-linear}). Since the tuple has $H$ components, this is trivially the case unless the system in question has infinite complexity. If $\Psi$ had infinite complexity, there would exist two forms in $\Psi$ that are affine-linearly dependent. However, that would imply that $(C_1h^d,C_1h^{d-1})=(C_2(h')^d,C_2(h')^{d-1})$ for some $h\neq h'$ and some constants $C_1,C_2$ not both zero, but this is clearly impossible.

Now, by the generalized von Neumann theorem in the form of~\cite[Lemma 9.6]{fh-jams}, the left-hand side of ~\eqref{e11b} is 
\begin{align}\label{eq12}
\leq C(H) N^{2}\max_{1\leq j<q}\|g^j\|_{U^{H}[N]}+o_{N\to \infty}(N^2)    
\end{align}
for some $C(H)\geq 1$. 

Recall that by assumption there is a function $\psi(N)$ tending to infinity slowly such that
\begin{align}\label{e30}
 \mathbb{D}(g^j,\chi;N)\geq \psi(N)    
\end{align}
for all $1\leq j\leq q-1$. Since $g$ takes values in the $q$th roots of unity, by Lemma~\ref{le_pret} we can upgrade~\eqref{e30} to a stronger form
\begin{align}\label{e160}
\inf_{|t|\leq N^2}\mathbb{D}(g^j,\chi(n)n^{it};N)\gg_{q} \psi(N).
\end{align}

By a result of Frantzikinakis and Host~\cite[Theorem 2.5]{fh-jams}, and~\eqref{e160} (which implies that $g^j$ is aperiodic), we have
\begin{align}\label{eq40}
\|g^j\|_{U^k[N]}=o_{k,N\to \infty}(1)    
\end{align}
for any fixed $k\in \mathbb{N}$ and all $1\leq j\leq q-1$. Now, since~\eqref{eq40} holds for every $k$, by a standard argument by contradiction there must exist some slowly growing function $K(N)$ such that~\eqref{eq40} holds \emph{uniformly} for $K\leq K(N)$ with decay rate $q^{-2K(N)}$, that is,
\begin{align}\label{e41}
\sup_{k\leq K(N)}\|g^j\|_{U^k[N]}\leq q^{-2K(N)}.    
\end{align}
We now take $H=H(N)$ be a function tending to infinity so slowly that $H(N)\leq K(N)$ and $C(H)\leq K(N)$, and assume without loss of generality that the $o_{N\to \infty}(N^2)$ term in~\eqref{eq12} is bounded by $N^2q^{-2H(N)}$. Comparing~\eqref{e11b},~\eqref{eq12} and~\eqref{e41}, this means that $\delta$ tends to $0$ as $N\to \infty$, contrary to the assumption that $\delta\gg 1$ for $N\in \mathcal{N}$. The proof is now complete.

(ii) Let $H=H(N)$ be a function tending to infinity sufficiently slowly. By the union bound, it suffices to show that for each fixed $v$ with $v^q=1$ we have $g(n^d+a)=v$ for some $n\leq H$, unless $a\in [1,N]$ belongs to an exceptional set of logarithmic density $0$. 

Let $\mathcal{E}_N$ be the set of $a\in [1, N]$ for which $g(n^d+a)\neq v$ for all $n\leq H$. Suppose that there exists a constant $\delta>0$ such that $\mathbb{E}_{a\leq N}^{\log}1_{\mathcal{E}_N}(a)\geq \delta $ for $N\in \mathcal{N}$, where $\mathcal{N}\subset \mathbb{N}$ is some infinite set. Then for $N\in \mathcal{N}$ we have 
\begin{align*}
\mathbb{E}^{\log}_{a\leq N}\left|\sum_{h\leq H}1_{g(h^d+a)\neq v}-(1-1/q)H\right|^2\gg_q \delta H^2.    
\end{align*}
Expanding out the square, we get
\begin{align*}
\sum_{1\leq h_1,h_2\leq H} \mathbb{E}^{\log}_{a\leq N}(1_{g(h_1^d+a)\neq v}-(1-1/q))(1_{g(h_2^d+a)\neq v}-(1-1/q))\gg_q \delta H^2.    
\end{align*}
The contribution of the terms $h_1=h_2$ is $\ll H$, so 
\begin{align*}
\sum_{1\leq h_1<h_2\leq H} \mathbb{E}^{\log}_{a\leq N}(1_{g(h_1^d+a)\neq v}-(1-1/q))(1_{g(h_2^d+a)\neq v}-(1-1/q))\gg_q \delta H^2.    
\end{align*}
Using the Fourier expansion~\eqref{eq13} and the pigeonhole principle, we see that 
\begin{align}\label{eq56}
\sum_{1\leq h_1<h_2\leq H}\mathbb{E}^{\log}_{a\leq N}g^{j_1}(a+h_1^d)g^{j_2}(a+h_2^d)\gg_{q}\delta H^2    
\end{align}
for some $1\leq j_1,j_2\leq q-1$ and for $N\in \mathcal{N}$. 

Now, by Tao's two-point logarithmic Elliott conjecture~\cite[Corollary 1.5]{tao-chowla} and~\eqref{e160}, the left-hand side of~\eqref{eq56} is $o_{N\to \infty}(H^2)$ (here we need the assumption that $H(N)$ tends to infinity arbitrarily slowly with $N$ to account for the lack of uniformity of the two-point correlation result with respect to the shifts $h_i$). But this contradicts the assumption that $\delta\gg 1$. Thus the claim of the theorem follows.
\end{proof}

\section{Multivariate polynomials}

\begin{proof}[Proof of Theorem~\ref{thm_multivariate}] (a) We claim that it suffices to show that for any $1\leq a,b\leq 100$ there exists \emph{one} coprime pair $(m,n)$ for which $\lambda(am^3+bn^4)=v$. Indeed, if $h$ is such that $\lambda(h)=v$ and $h=am^3+bn^4$ for some coprime $(m,n)$, then the Diophantine equation 
\begin{align}\label{eq15}
hz^2=am^3+bn^4     
\end{align}
has a solution $(z,m,n)\in \mathbb{N}^3$ with the greatest common divisor of $m,n,z$ being $1$. The Diophantine equation~\eqref{eq15} is a generalized Fermat equation, and since the sum of the inverses of the exponents satisfies $1/2+1/3+1/4>1$, a result of Beukers~\cite{beukers} says that the existence of one nonzero coprime solution to~\eqref{eq15} implies the existence of infinitely many such solutions (in fact, the solutions can be parametrized in terms of polynomials). Consequently, whenever we have we have $\lambda(am^3+bn^4)=v$ for one coprime pair $(m,n)$, we have the same for infinitely many coprime pairs $(m,n)$. 

To find one coprime solution $(m,n)$ to $\lambda(am^3+bn^4)=v$ for each $1\leq a,b\leq 100$ and each $v\in \{-1,+1\}$, we perform a brute force search. It certainly suffices to find for each $1\leq a,b\leq 100$ two integers $m_{+}$, $m_{-}$ such that $\lambda(am_{\pm}^3+b)=\pm 1$. Performing a very quick numerical search using Python\footnote{\label{note1}The Python 3 code for this computation can be found along with the arXiv submission of this paper.}, we see that the desired $m_{+},m_{-}$ always exist in $[1,20]$. 

(b) The proof of part (b) is similar to that of part (a), since Beukers's result applies to the equation $hz^2=am^2+bn^k$ as well (as $1/2+1/2+1/k>1$). Thus, we simply need to show that $\lambda(am_{\pm}^2+b\cdot 1^k)=\pm 1$ has at least one solution for each choice of the $\pm$ sign and for each $1\leq a,b\leq 100$ and $k\geq 1$. Clearly, the value of $k$ makes no difference, as we have specialized to $n=1$. Finding these $m_{\pm}$ is again a very quick numerical search using Python (see footnote~\ref{note1}), and suitable $m_{+},m_{-}$ turn out to always exist in the range $[1,30]$.
\end{proof}

\bibliography{polyrefs}
\bibliographystyle{plain}

\end{document}